\documentclass[12pt,twoside]{article}
\usepackage{amsmath,amsfonts,amssymb,a4,epsfig,array,psfrag}
\usepackage{enumerate}
\usepackage[normalem]{ulem}
\usepackage{url}
\usepackage{amsthm}
\usepackage{etoolbox} 

\usepackage{subfig}
\usepackage{color}
\usepackage[symbol]{footmisc}

\usepackage{tikz}
\usetikzlibrary{calc,3d,patterns}
\usetikzlibrary{decorations.markings}
\definecolor{verylight}{gray}{.85}
\definecolor{light}{gray}{.75}
\definecolor{lightish}{gray}{.6}
\definecolor{medium}{gray}{.5}
\definecolor{dark}{gray}{.25}
\definecolor{kindofdark}{gray}{.4}
\definecolor{posetshade}{gray}{.85} 

\tikzset{->-/.style={decoration={markings, mark=at position #1 with {\arrow{>}}},postaction={decorate}}}

\graphicspath{{figures/}}

\newtheorem{theo}{Theorem}

\newtheorem{prop}{Proposition}[section]
\newtheorem{coro}[prop]{Corollary}

\newtheorem{lemma}[prop]{Lemma}

\newtheorem{definition}[prop]{Definition}

\theoremstyle{definition}
\newtheorem{example}[prop]{Example}
\AtEndEnvironment{example}{\null\hfill$\diamond$}%
\newtheorem{remark}[prop]{Remark}
\AtEndEnvironment{remark}{\null\hfill$\diamond$}%

\theoremstyle{remark}

\newcommand{\basetiling}{\oplus}

\newcommand{\interior}{\operatorname{int}}

\newcommand{\Flux}{\operatorname{Flux}}

\newcommand{\NN}{{\mathbb{N}}}
\newcommand{\ZZ}{{\mathbb{Z}}}

\newcommand{\RR}{{\mathbb{R}}}

\newcommand{\Ss}{{\mathbb{S}}}

\newcommand{\cT}{{\cal T}}

\newcommand{\cL}{{\cal L}}

\newcommand{\cG}{{\cal G}}

\newcommand{\Tw}{\operatorname{Tw}}
\newcommand{\TW}{\operatorname{TW}}

\newcommand{\ccolor}{\operatorname{color}} 

\newcommand{\wind}{\operatorname{wind}}

\newcommand{\Link}{\operatorname{Link}}

\date{\today}

\begin{document}
\title{On the connectivity of spaces of three-dimensional domino tilings}
\author{Juliana Freire \and Caroline Klivans \\ \and Pedro H. Milet \and Nicolau C. Saldanha}

\maketitle

\begin{abstract}
We consider domino tilings of three-dimensional cubiculated manifolds
with or without boundary,
including subsets of Euclidean space and three-dimensional tori.
In particular, we are interested
in the connected components of the space of tilings of such regions
under local moves.
Building on the work of the third and fourth authors~\cite{segundoartigo},
we allow two possible local moves, the \emph{flip} and
\emph{trit}. 
These moves are considered with respect to two topological
invariants, the \emph{twist} and \emph{flux}.

Our main result proves that, up to refinement, \\
$\bullet$  Two tilings are connected by flips and trits
if and only if they have the same flux. \\
$\bullet$   Two tilings are connected by flips alone 
if and only if  they have the same flux and twist.
\end{abstract}

\section{Introduction}

\footnotetext{2010 {\em Mathematics Subject Classification}.
Primary 05B45; Secondary 52C20, 52C22, 05C70.
{\em Keywords and phrases} Three-dimensional tilings,
dominoes,
dimers,
flip accessibility,
connectivity by local moves}

Tiling problems have received much attention in the second half of the
twentieth century: two-dimensional domino and lozenge tilings in particular,
due to their connection to the dimer model and to matchings in a graph. A large
number of techniques have been developed for solving various problems in two
dimensions. For instance, Kasteleyn~\cite{Kasteleyn19611209},
Conway and Lagarias~\cite{conway1990tiling}, Thurston~\cite{thurston1990}, Cohn, Elkies, Jockush, Kuperberg, Larsen, Propp and Shor~\cite{jockusch1998random,cohn1996local,elkies1992alternating}, and Kenyon,
Okounkov and Sheffield~\cite{kenyonokounkov2006dimers,kenyonokounkov2006planar} have used
 very interesting techniques, ranging from abstract algebra to probability.

A number of generalizations of these techniques have been made to the
three-dimensional case.  Randall and Yngve \cite{randall2000random}
considered tilings of ``Aztec'' octahedral and tetrahedral regions
with triangular prisms, which generalize domino tilings to three
dimensions. 
Linde, Moore and Nordahl
\cite{linde2001rhombus} considered families of tilings that generalize
rhombus (or lozenge) tilings to arbitrary dimensions. 
 Bodini \cite{bodini2007tiling} considered tiling problems of
pyramidal polycubes.  And, while the problem of counting domino tilings is known
to be computationally hard (see \cite{pak2013complexity}), some
asymptotic results, including for higher dimensions, date as far back as
1966 (see \cite{hammersley1966limit,ciucu1998improved,friedland2005theory}).

Most relevant to the discussion in this paper are the problems of
connectivity of the space of tilings under local moves.  A \emph{flip}
is the simplest local move: remove two adjacent parallel dominoes and
place them back in the only possible different position.  In two
dimensions, any two tilings of a simply connected region are flip
connected; see e.g., \cite{thurston1990,saldanhatomei1995spaces}.
This is no longer the case when one considers tilings in three dimensions.

 Even for simple three-dimensional regions, the space of tilings is no longer connected by
flips.  This is perhaps not surprising as the flip is inherently a
 two-dimensional move. 
The \emph{trit} is a three-dimensional local
move, which lifts three dominoes sitting in a $2 \times 2 \times 2$
cube, no two of which are parallel, and places them back in the only
other possible configuration (see Figure~\ref{fig:postrit}).  It is
natural to ask if these two moves, the flip and trit are enough to
connect all tilings of three-dimensional spaces.  In general, the
answer is again no.
Here we consider connectivity of tilings
taking into account two topological invariants. 
In doing so, we are able to characterize (up to refinement)
when two tilings are connected by flips or flips and trits.

The first invariant is the \emph{Flux} of a tiling. 
The flux of a tiling of a region $R$ takes values in the first homology group
$H_1(R;\ZZ)$. 
The second invariant is the \emph{twist} of a tiling. 
The twist assumes values either in $\ZZ$ or in $\ZZ/m\ZZ$
where $m$ is a positive integer depending on
the value of the Flux.
For contractible regions (such as boxes),
the twist assumes values in $\ZZ$.
If $R$ is a torus of the form $\ZZ^3 / \mathcal{L}$, where
$\mathcal{L}$ is spanned by $(a,0,0),(0,b,0),(0,0,c)$
(with $a,b,c$ even positive integers)
then the twist assumes values in $\ZZ$ if Flux is $0$
and in some $\ZZ/m\ZZ$ otherwise.
The twist was first introduced by
Milet and Saldanha~\cite{segundoartigo, primeiroartigo} for particularly
nice regions.  In that context, the twist has a simple combinatorial
definition.  Unfortunately, it does not extend to the more general
tiling domains considered here.

Our new definition of twist, and our introduction of Flux 
relies on the construction of auxiliary
surfaces.  The difficulty is that the required surfaces may not always exist.
The difficulty is addressed by using the concept of
\emph{refinement}.  A region $R$ is refined by decomposing each cube
 of $R$ into $5 \times 5 \times 5$ smaller cubes.
Refinement
guarantees the existence of auxiliary surfaces which, borrowing from
knot theory, we call \emph{Seifert surfaces}.

Informally, Flux measures how a tiling flows across a surface
boundary. 
 If two tilings are flip and trit connected, they must have
equal Flux.  The twist measures how ``twisted'' a tiling is by trits:
under a trit move the twist changes by exactly one. 
A key property is that if two tilings 
are in the same flip connected component,
then they must have equal twist. 
The converse is false in general.

The twist can also be interpreted as a discrete analogue of
\emph{helicity} arising in fluid mechanics and topological
hydrodynamics, see
e.g.~\cite{moffatt1969degree, arnold1999topological, khesin2005topological};
the authors thank Yuliy Baryshnikov for bringing this concept to our attention. 
The helicity of a vector field on a domain in
$\mathbb{R}^3$ is a measure of the self linkage of field lines.
An important recent result shows that helicity
is the only integral invariant of
volume-preserving transformations~\cite{helicityPNAS}.

Our main result is a characterization
of the connectedness of three-dimensional tilings
by flips and trits with respect to flux and twist.

\begin{theo}
\label{theo:main}
Consider a cubiculated region $R$ and two tilings $t_0$ and $t_1$ of $R$.
\begin{enumerate}
\item[(a)]{There exists 
a sequence of flips and trits taking a refinement of $t_0$
to a refinement of $t_1$
if and only if $\Flux(t_0) = \Flux(t_1)$.}
\item[(b)]{There exists 
a sequence of flips taking a refinement of $t_0$ to a refinement of $t_1$
if and only if $\Flux(t_0) = \Flux(t_1)$ and $\Tw(t_0) = \Tw(t_1)$.}
  
\end{enumerate}
\end{theo}

In general, the refinement condition is necessary
in the statement of the theorem.
However, it is not known if the refinement condition
may be dropped in certain special cases.
For nice regions (such as boxes) there is empirical evidence,
see~\cite{segundoartigo, saldanha2021}, 
that refinement is almost never necessary;
for item (a) it may never be necessary.

Section~\ref{sec:prelim} contains preliminaries
for the regions we will consider. 
The two local moves, the flip and trit,
are introduced in Section~\ref{sec:local}. 
Section~\ref{sec:Flux} introduces the flux.
In Sections~\ref{sec:surface} and~\ref{sec:fluxthroughsurfaces}
we work heavily with discrete surfaces
leading to the definition of the twist in Section~\ref{sec:twist}. 
Section~\ref{sec:coquadriculated} extends
the concept of height functions to our setting
where they are better described as height forms. 
Theorem~\ref{theo:main} is proved in Section~\ref{sec:proof}. 
We end with a discussion of further questions and conjectures
concerning three-dimensional tilings.

The authors are thankful for the generous support of
CNPq, CAPES, FAPERJ and a grant from the  Brown-Brazil initiative.  
They would also like to thank the referee 
for several helpful and insightful comments.

\section{Preliminaries}
\label{sec:prelim}
\subsection{Cubiculated Regions}
\label{sec:regions}

In this paper, we  consider tilings of certain three-dimensional regions. 
By a \emph{cubiculated region} $R$, we will mean a cubical complex embedded as a finite polyhedron in $\RR^N$, for some $N$, 
which is also a connected oriented topological manifold of dimension three
with (possibly empty) boundary $\partial R$.
We assume that:
(i) interior edges of $R$ are surrounded by precisely four cubes; 
(ii) cubes are painted  {\it black} or  {\it white}
such that two adjacent cubes have opposite colors; and
(iii)  the number of black cubes equals the number of white cubes. 
It follows from this definition that $\partial R$ is also a polyhedron and
an oriented topological manifold of dimension two.

A  {\it domino}
is the union of two adjacent cubes in $R$
and a  {\it (domino) tiling} of $R$ is a collection of dominoes
with disjoint interior whose union is $R$.

\begin{example}
\label{example:firsttilingsbox}
A {\it box} is a cubiculated region  $[0,L]\times[0,M]\times[0,N]$,
where $L$, $M$ and $N$ are positive integers,
at least one of them even.
Figure \ref{fig:233} shows different representations 
of the $3\times 3\times 2$ box.
In one representation, a tiling is shown
as a pile of blocks in perspective:
this representation is very intuitive but not practical
except for very small examples.
The other representation is by floors:
we draw each floor separately (from bottom to top).
Horizontal dominoes are contained in a single floor
and are represented by rectangles.
Vertical dominoes appear on two adjacent floors as squares:
we mark the bottom half of a vertical domino by shading it.
\end{example}

\begin{figure}[ht]
\centering
\centerline{\includegraphics[scale=0.25]{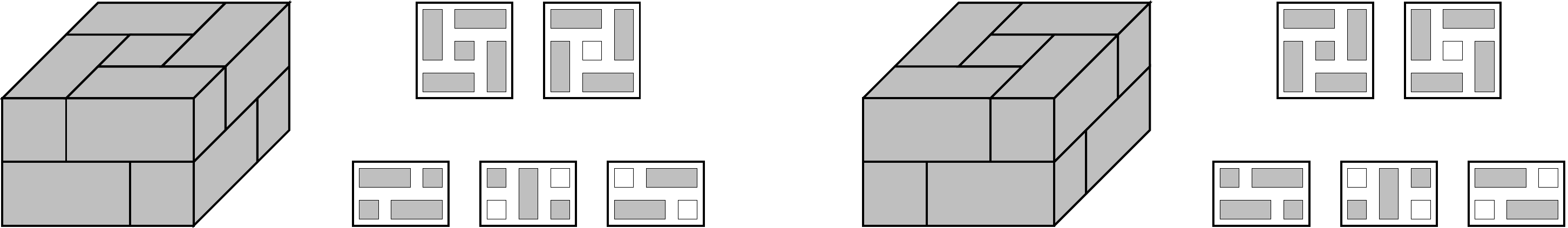}}
\caption{The two tilings of a $3 \times 3 \times 2$ box
that have no possible flips. 
Each tiling is shown three times:
in 3D perspective,
by floors with a $3\times 3$ base
and again by floors but rotated,
so that now the base is a $3\times 2$ rectangle.
Note that the tilings differ only in that
the two floors have been swapped. }
\label{fig:233}
\end{figure}

\begin{example}
\label{example:firsttilingstorus}
Let $\cL \subset \ZZ^3$ be a three-dimensional lattice such that
$(x,y,z) \in \cL$ implies $x+y+z$ is even. 
The {\it torus}  $T = \RR^3/\cL$
is a cubiculated region without boundary.
Figure \ref{fig:666f0} shows the $6\times 6\times 6$ torus, i.e.,
$\cL$ is generated by $6e_1$, $6e_2$ and $6e_3$.
We also draw the tiling by floors.
Opposite sides should be identified.
\end{example}

\begin{figure}[ht]
\centerline{\includegraphics[scale=0.25]{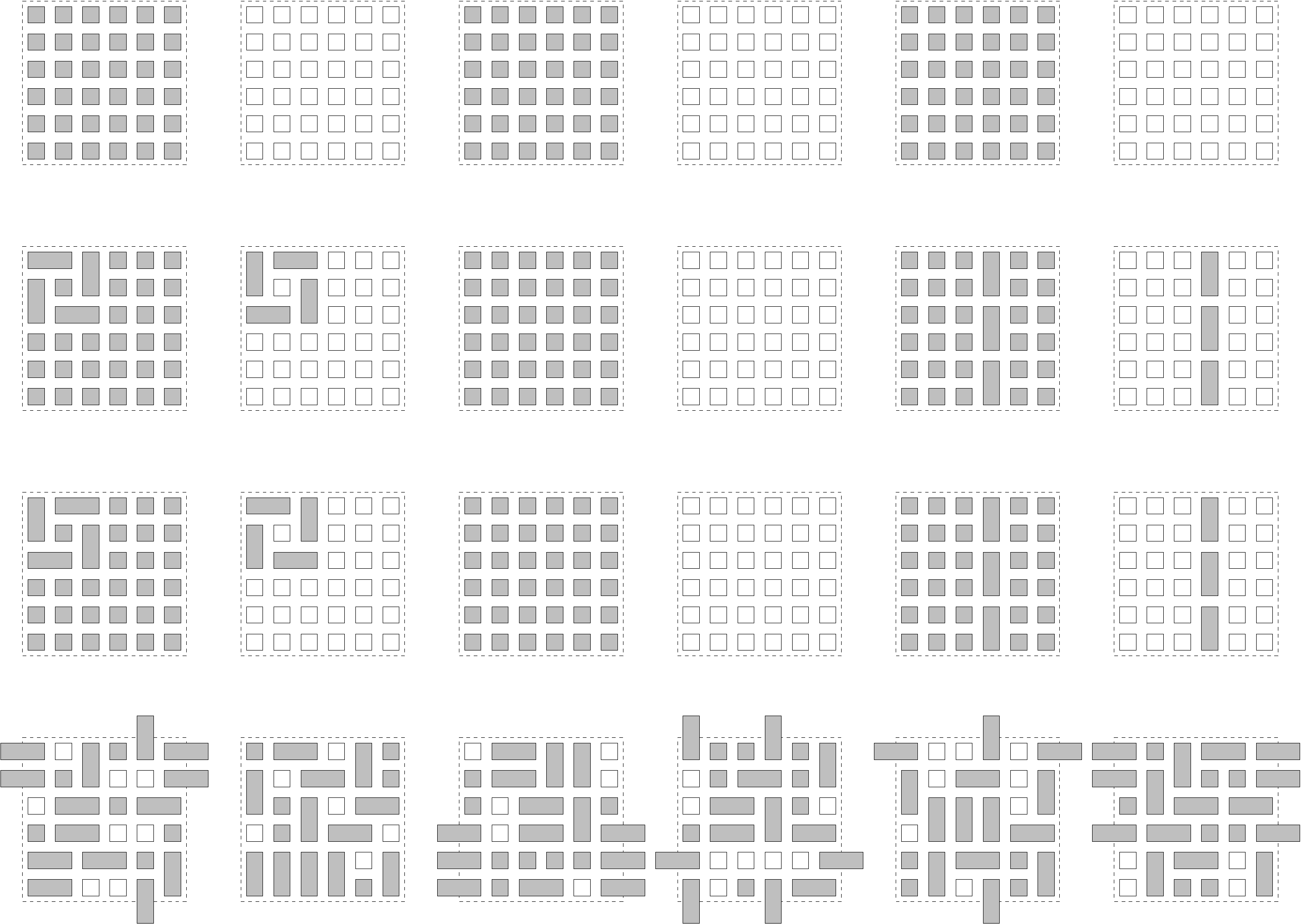}}
\caption{Four tilings of the $6\times 6\times 6$ torus,
one per row.
Looking forward to Section \ref{sec:Flux},
the tiling $t_\basetiling$ on the first row
is designated as the base tiling.
Notice that all dominoes in $t_{\basetiling}$
are vertical (in the direction $e_3$).}
\label{fig:666f0}
\end{figure}

We also consider the dual cubical complex $R^\ast \subset R$.
Vertices of $R^\ast$ are centers of cubes in $R$ and
edges of $R^\ast$ join centers of adjacent cubes in $R$.
There is a cube in $R^\ast$ around each interior vertex of $R$:
its eight vertices are the centers of the cubes in $R$ adjacent
to the interior vertex.
The dual cubical complex $R^\ast$ may or may not be a manifold,
for instance, there may exist edges of $R^\ast$ not adjacent
to any cube of $R^\ast$.
Figure \ref{fig:233ast} shows $R$ and $R^\ast$
for the $3\times 3\times 2$ box 
and two tilings in the dual cubical complex $R^\ast$.

\begin{figure}[ht]
\centerline{\includegraphics[scale=0.2]{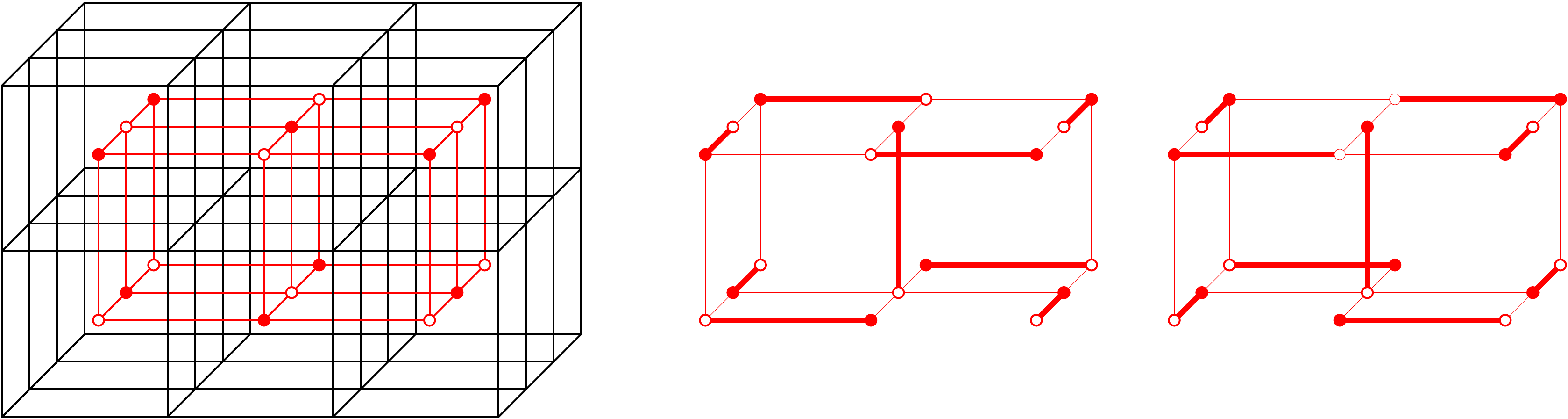}}
\caption{The $3\times 3\times 2$ box $R$ and its dual $R^\ast$:
notice that drawing both $R$ and $R^\ast$ quickly becomes impractical.
The two tilings in Figure \ref{fig:233} are shown in $R^\ast$
as matchings of $\cG(R^\ast)$.}
\label{fig:233ast}
\end{figure}

We will also work with the graph $\cG(R)$ and its dual $\cG(R^\ast)$. 
The vertices and edges of $\cG(R)$ are just the vertices and edges of $R$;
in other words, $\cG(R)$ is the $1$-skeleton of $R$. 
Similarly,
$\cG(R^\ast)$ is the $1$-skeleton of $R^\ast$ which is a bipartite graph.
A tiling of $R$ is equivalent to a matching of the graph $\cG(R^\ast)$.
When seen as an edge in $\cG(R^\ast)$, a domino is called a  {\it dimer}.
Tiling regions are often simply regarded as subgraphs of $\cG(R^\ast)$.
It is important to note that the regions we are working with here must
be topological manifolds,
therefore it does not suffice to consider
arbitrary subgraphs of the $\mathbb{Z}^3$ lattice.


\subsection{Cycles}

An important concept to have in mind throughout this paper is
the interpretation of the difference of two tilings as
a union of disjoint cycles. 

An  {\it embedded cycle} is an injective continuous map $\gamma: \Ss^1
\to R^\ast \subset R$ whose image is a union of vertices and edges of
$R^\ast$.  Thus, an embedded cycle is a cycle in the graph theoretical
sense for $\cG(R^\ast)$. 

 We will also consider cycles homologically as elements of
 $Z_1(R^\ast;\ZZ)$, the kernel of the boundary map from one to zero
 dimensional cells of $R^\ast$.
Similarly, since a dimer connects a pair of vertices of opposite color,
we may also think of dimers as oriented edges
pointing from the center of a white cube to the center of a black cube
(by convention), i.e., as generators of $C_1(R^\ast;\ZZ)$,
the one dimensional chain group.

With this point of view in mind, given two tilings $t_0$ and $t_1$,
we define $t_1 - t_0$ to be the union of the dimers in both tilings,
with the dimers in $t_0$ having their orientations reversed.
Hence, $t_1 - t_0$ is the union of disjoint cycles;
cycles of length $2$ in the graph theoretical sense
are called  {\it trivial cycles} and are usually ignored.
Figure \ref{fig:233cyc} shows a simple example.

\begin{figure}[ht]%
\centering
\centerline{\includegraphics[scale=0.333333]{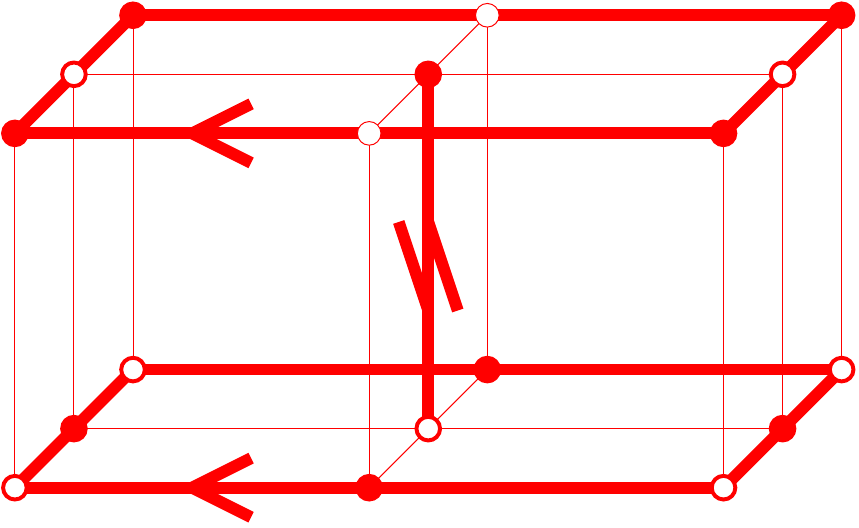}}
\caption{The two tilings in Figures \ref{fig:233} and \ref{fig:233ast},
now plotted together as a union of oriented dimers.
This yields three curves, one of which is trivial.}%
\label{fig:233cyc} 
\end{figure}

\subsection{Refinements}
\label{sec:refine}

A region $R$ is  {\it refined} by decomposing each cube of $R$
into $5 \times 5 \times 5$ smaller cubes;
the corners and the center are painted the same color as the original cube.
This defines a new cubiculated region $R'$.
We sometimes need to refine a region $R$ not once but $k$ times:
we then call the resulting region $R^{(k)}$.
As topological spaces, $R$ and $R'$ are equal.
A tiling $t$ is refined by decomposing each domino of $t$
into $5 \times 5 \times 5$ smaller dominoes,
each one parallel to the original domino.
Again, this defines a new tiling $t'$;
if we refine $k$ times we obtain $t^{(k)}$. 
For simplicity, in moving to the dual,
we write $R^{(k) \ast}$ instead of $(R^{(k)})^{\ast}$.  

Some comments on the choice of $5 \times 5 \times 5$ are in order. 
Dividing each cube into $2 \times 2 \times 2$
smaller cubes would erase the distinction between black and white
 and make the entire discussion trivial.
Dividing into $3\times 3\times 3$ smaller cubes works for our purposes but 
 the fact that the central cube in this small block
has the opposite color as the corner is a source of unnecessary confusion.
What we need, therefore,
is a positive integer greater than $1$
which is congruent to $1$ mod $4$;
$5$ being the smallest.

\section{Local Moves: Flips and Trits}
\label{sec:local}

A  {\it flip} is a move that takes a tiling $t_0$ into another tiling $t_1$
by removing two parallel dominoes that form a $2 \times 2 \times 1$ ``slab''
and placing them back in the only other position possible.
The  {\it flip connected component} of a tiling $t$ is the set of all
tilings that can be reached from $t$ after a sequence of flips.  For
instance, the $3 \times 3 \times 2$ box has $229$ total tilings, but
only three flip connected components. 
Two components contain just one tiling each;
i.e. there are no possible flip moves from these tilings,
see Figure~\ref{fig:233}.
The $4 \times 4 \times 2$ box admits $32,000$ tilings.
Under flips, it has $9$ connected components.
The largest one consists of all tilings with twist $0$
and has $31,484$ tilings.
There are $4$ components with $128$ tilings each 
and $4$ components consisting of a single tiling each.
Figure \ref{fig:442} shows two of these isolated tilings.

\begin{figure}[ht]
\centerline{\includegraphics[scale=0.2]{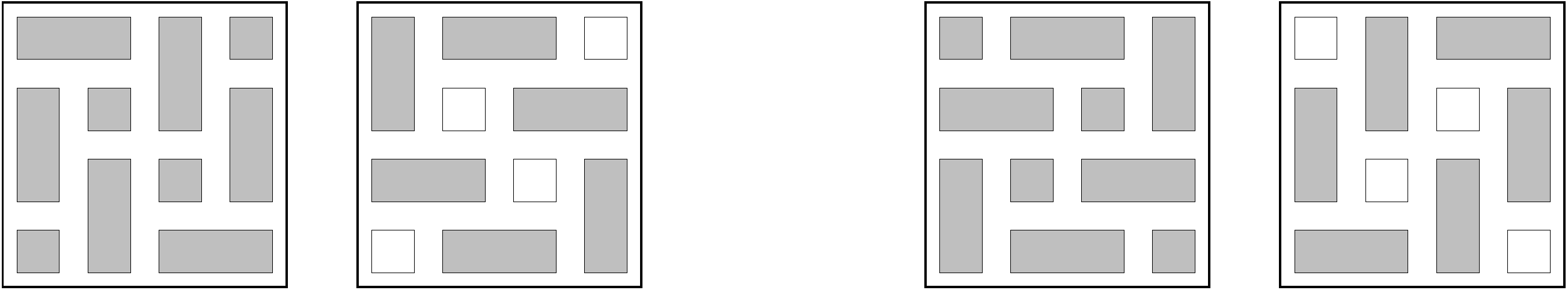}}
\caption{Two tilings $t_0$ and $t_1$ of the $4\times 4\times 2$ box.
Neither admits a flip. 
Their refinements are mutually accessible by flips.}
\label{fig:442}
\end{figure}

A  {\it trit} is a move involving three dominoes
which sit inside a $2 \times 2 \times 2$ cube
where each domino is parallel to a distinct axis. 
We thus necessarily have some rotation of Figure \ref{fig:postrit}.
The trit that takes the drawing at the left of Figure \ref{fig:postrit}
to the drawing at the right is a  {\it positive trit}.
The reverse move is a  {\it negative trit}.
Notice that the $2 \times 2 \times 2$ cube
need not be entirely contained in $R$.
On the other hand, not just the six cubes directly involved in the trit
but also at least one of the other two must be contained in $R$
otherwise $R$ is not a manifold.

\begin{figure}[ht]
\centering
\centerline{\includegraphics[width=2.5in]{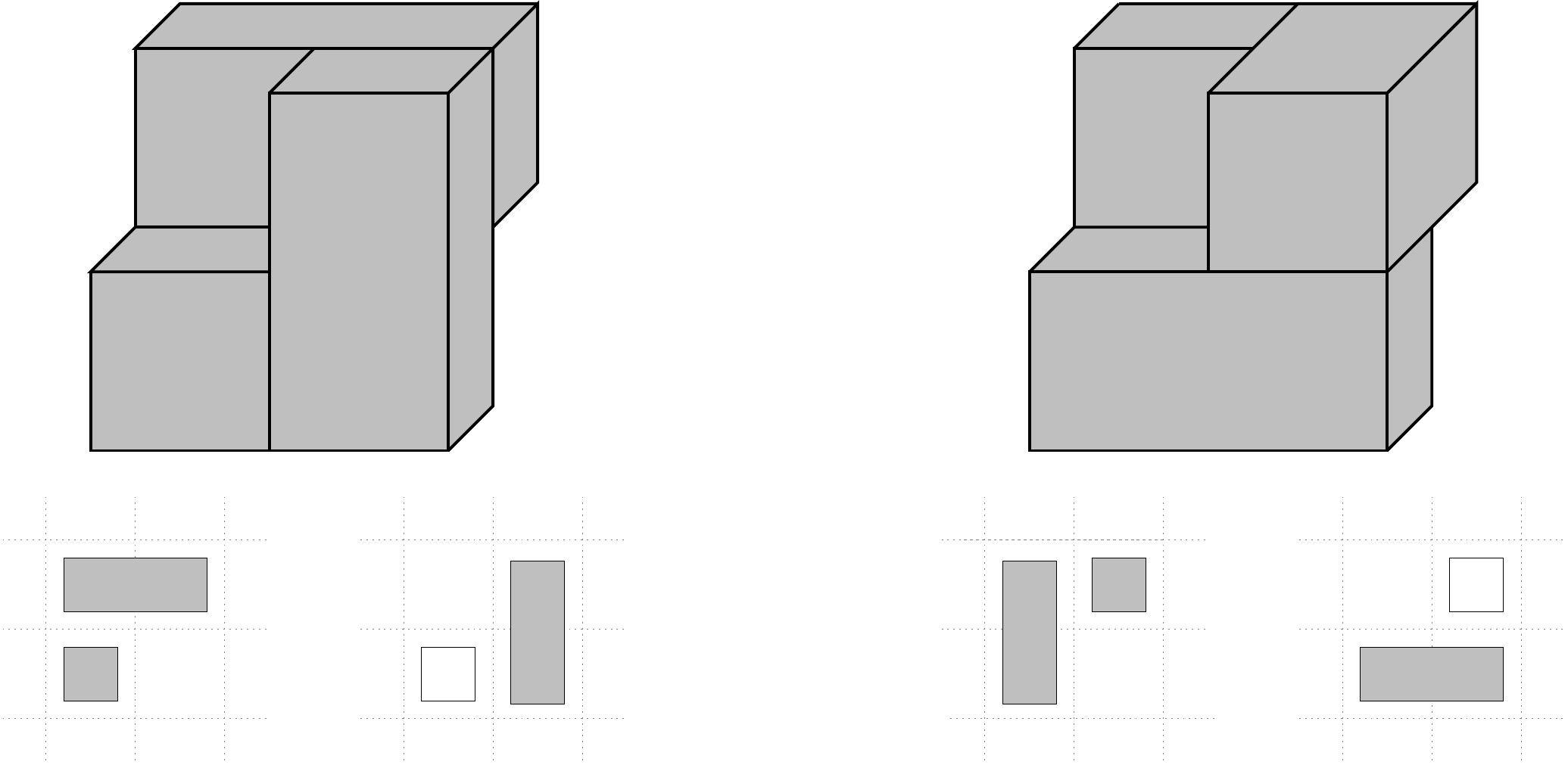}}
\caption{The anatomy of a positive trit (from left to right),
with drawings in perspective above floor diagrams.
The trit that takes the right drawing to the left one is a negative trit.
The empty corners  may represent either partial dimers that are not contained
in the $2 \times 2 \times 2$ cube or cubes that are not contained in the region (for instance, if the region happens not to be a box).}%
\label{fig:postrit}%
\end{figure}

Flips and trits behave well with respect to refinement.  
If $t_0$ and $t_1$ differ by a flip,
then their refinements differ by a sequence of $125$ flips.
If $t_0$ and $t_1$ differ by a trit,
then their refinements differ by a a trit and a sequence of flips. 
Therefore we have:

\begin{prop}
If $t_0$ and $t_1$
are connected by flips (resp. flips and trits)
then their refinements
are also connected by flips (resp. flips and trits).
\end{prop}

The converse does not hold however. 
The two tilings $t_0$ and $t_1$
of the $4\times 4\times 2$ box shown in Figure \ref{fig:442}
admit no flips and are therefore not mutually accessible by flips;
their refinements $t_0^{(1)}$ and $t_1^{(1)}$ are mutually accessible.

The tiling of the $8\times 8\times 4$ torus in Figure \ref{fig:488}
admits no flips or trits.
We do not know whether there exist tilings of boxes
which do not admit either flips or trits
(small boxes admit no such tilings).

\begin{figure}[ht]
\centerline{\includegraphics[scale=0.25]{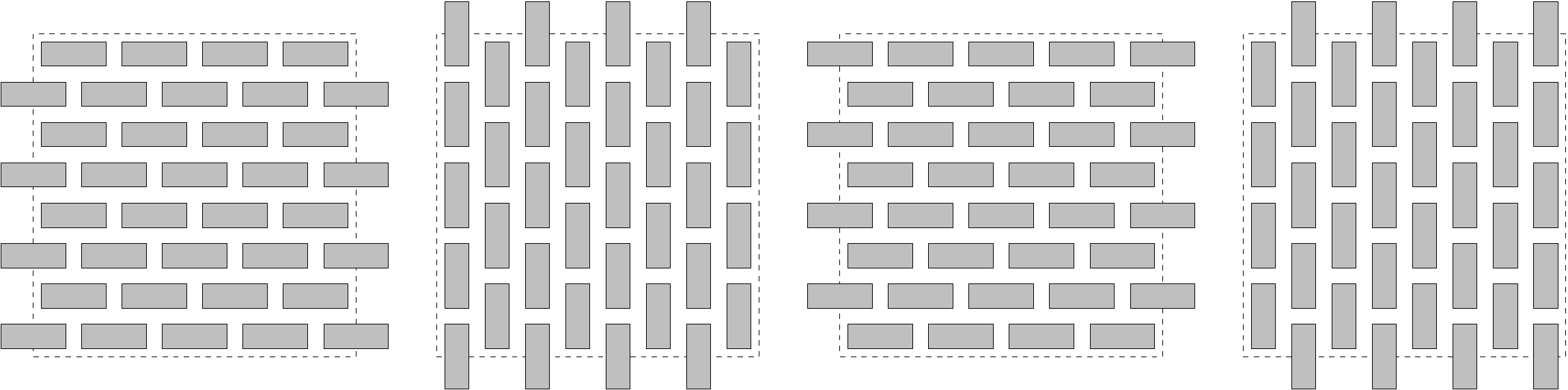}}
\caption{A tiling of the $8\times 8\times 4$ torus.
This tiling has flux $0$, twist $0$ and admits no flips or trits.}
\label{fig:488}
\end{figure}

\section{Flux}
\label{sec:Flux}

In this section we define and develop the concept of \emph{Flux} 
via homology theory,
a related notion of \emph{flux through a surface}
will be give in Section~\ref{sec:fluxthroughsurfaces}.

Recall that if $t_0$ and $t_1$ are two tilings,
then $t_1 - t_0 \in Z_1(R^*; \mathbb{Z})$, and the equivalence class
$[t_1 - t_0]$ is an element of the homology group $H_1(R^*; \mathbb{Z})$
(which is naturally identified with $H_1(R;\mathbb{Z})$
since the inclusion $R^\ast \subset R$ is a homotopy equivalence).

\begin{definition}
Let $t_\basetiling$ be a fixed base tiling.  
The \emph{Flux} of a tiling $t$ is defined as:
\[ \Flux(t) = [t - t_\basetiling] \in H_1(R^*; \mathbb{Z}). \]
\end{definition}

\begin{example}
\label{example:flux}
If $R$ is a box then $\Flux(t) = 0$ for any tiling of $R$.
A more interesting example is a 3D torus $R = \RR^3/\cL$,
where the lattice $\cL \subset \ZZ^3$ is generated by
$L e_1$, $M e_2$ and $N e_3$ with
$L, M, N \in \NN^\ast$ and $N$ even.
Clearly, $H_1(R; \ZZ) = \ZZ^3$.

We choose as base tiling the \textit{vertical tiling},
with all dominoes in the direction $e_3$,
as in the first tiling of Figure \ref{fig:666f0}.
The tiling in Figure \ref{fig:488} has Flux $0$.
The four tilings of the $6\times 6\times 6$ torus
in Figure \ref{fig:666f0} have Flux equal to $(0,0,0)$;
the two tilings in Figure \ref{fig:666f1} have Flux equal to $(1,0,0)$.

In general, for the torus $R = \RR^3/\cL$
where $\cL$ is generated by $Le_1$, $Me_2$ and $Ne_3$ with even $N$
we take as base tiling the tiling similar
to the first one in Figure \ref{fig:666f0},
with all dominoes vertical.
Thus, for instance, the tiling in Figure \ref{fig:488}
has Flux equal to $(0,0,0)$.
\end{example}


\begin{figure}[ht]
\centerline{\includegraphics[scale=0.25]{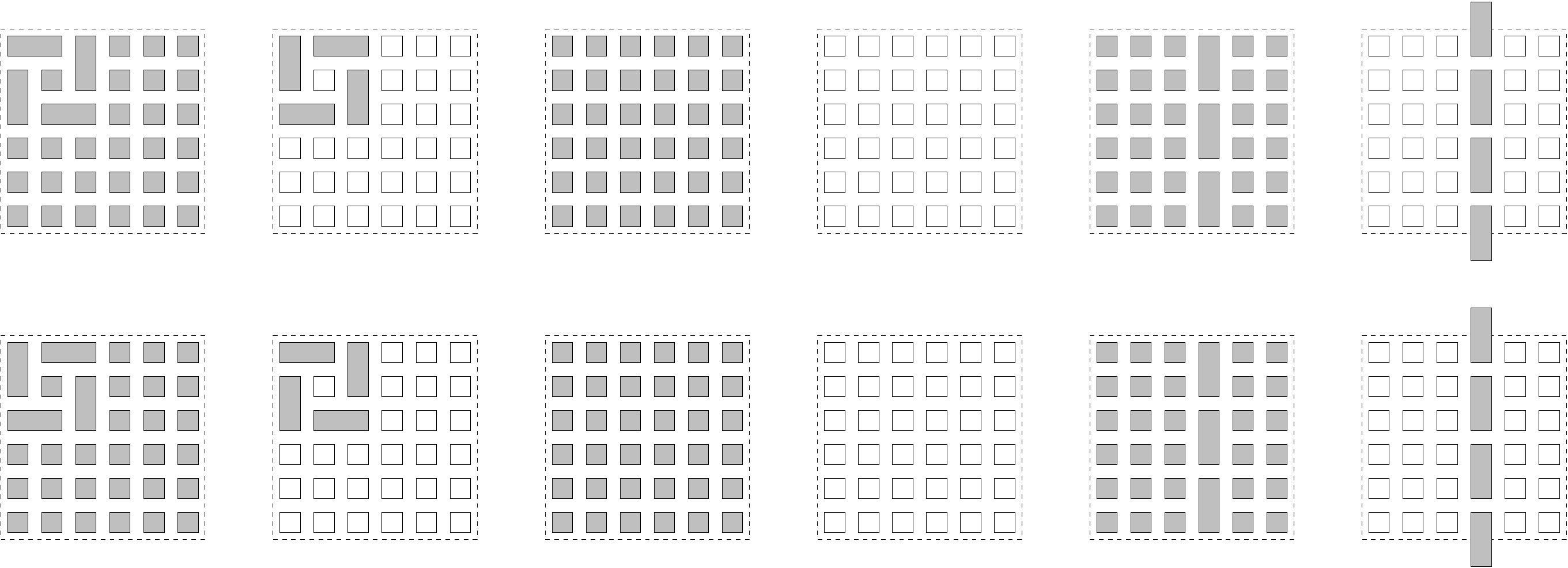}}
\caption{Two tilings $t_0$ of the $6\times 6\times 6$ torus
with Flux equal to $(1,0,0)$.}
\label{fig:666f1}
\end{figure}

If $t_0$ and $t_1$ differ by a flip
then $t_1 - t_0$ is the boundary of a square.
Similarly, if $t_0$ and $t_1$ differ by a trit
then $t_1 - t_0$ is the boundary of a sum of three squares.
In either case we have:

\begin{prop}
If $t_0$ and $t_1$ differ by flips and trits, then $\Flux(t_0) = \Flux(t_1)$.
\end{prop}

Again, the converse is not true.
The tiling in Figure \ref{fig:488} has Flux equal to $(0,0,0)$
but is not connected by flips and trits to the base tiling
(since it admits neither flips not trits).
Experiments, however, indicate that most tilings $t$ with
$\Flux(t) = \Flux(t_\basetiling) = 0$
can be joined to $t_\basetiling$ by a finite sequence of flips and trits.
Indeed, all four tilings in Figure \ref{fig:666f0}
can be pairwise joined by flips and trits.
Out of several thousand randomly selected tilings of tori,
not a single exception appeared.
The tiling $t_1$ in Figure \ref{fig:488} was specially constructed and is an exception.


While the converse is not true,
our main Theorem, Theorem~\ref{theo:main}, provides a correct converse to this statement:
there exist refinements of the two tilings
that are connected by flips and trits
if and only if the Flux are equal.

First, we show that the Flux is preserved under refinement.

\begin{lemma}
\label{lem:refineflux}
If $t'$ is the refinement of a tiling $t$, then 
$\Flux(t') = \Flux(t)$.
\end{lemma}

Before proving the lemma, we introduce the concept
of a \textit{modified refinement} $t^{(k;\gamma)}$
of a tiling $t$: this construction will be used again later.
An embedded cycle $\gamma: \Ss^1 \to R^\ast$ is refined
by merely interpreting it as $\gamma' = \gamma: \Ss^1 \to (R')^\ast$.
A tiling $t$ of $R$ is  {\it tangent} to an embedded cycle $\gamma$
if every vertex of $R^\ast$ in the image $\gamma[\Ss^1]$
is one of the endpoints of a dimer $d$
contained in the tiling $t$ and in $\gamma[\Ss^1]$.
Unfortunately, if $t$ is tangent to $\gamma$
it does not follow that $t'$ is tangent to $\gamma'$.
In this situation, we therefore define
a  {\it modified refinement} $t^{(1;\gamma)}$
which is tangent to $\gamma$.
At each $5 \times 5 \times 5$ cube around a vertex of $R^\ast$
belonging to the image of $\gamma$,
perform a flip in $t'$ if needed
(as in Figure~\ref{fig:refine})
to obtain the required tiling $t^{(1;\gamma)}$
tangent to $\gamma'$.
We can iterate this procedure to define $t^{(k;\gamma)}$
which is tangent to $\gamma^{(k)}$
and connected to $t^{(k)}$ by flips in the neighborhood of $\gamma$.

\begin{figure}[ht]
\centerline{\includegraphics[scale=0.25]{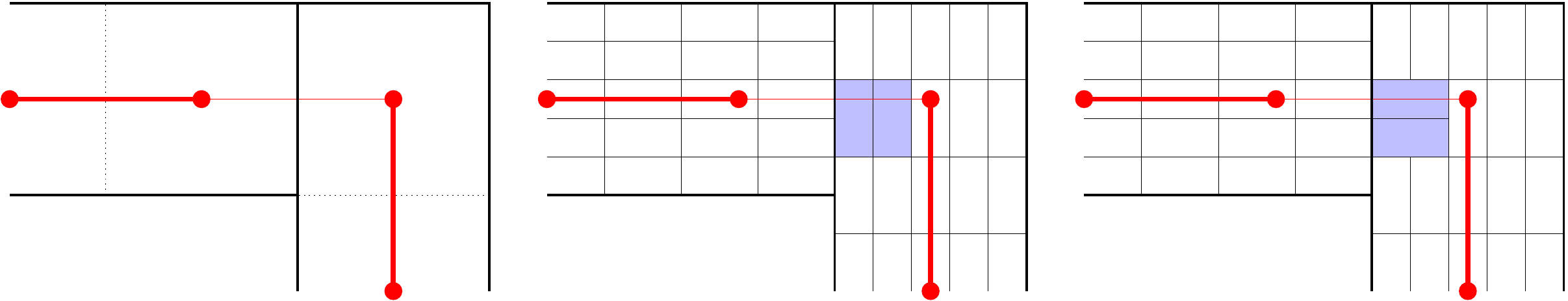}}
\caption{Detail of a tiling $t$ (left)
and of its a refinement $t'$ (center).
Given a boundary cycle $\gamma$ (shown in all three figures),
a few flips take us from $t'$ to
a \textit{modified refinement}:
a tiling $t^{(1,\gamma)}$ (right) which is tangent to $\gamma$.
The position where flips were performed is shaded.}
\label{fig:refine} 
\end{figure}

\begin{proof}[Proof of Lemma \ref{lem:refineflux}]
Consider two tilings $t_0$ and $t_1$ of $R$
and let $\gamma$ be the system of cycles $t_1 - t_0$.
Then $t_0^{(k;\gamma)}$ and $t_1^{(k;\gamma)}$
are both tangent to $\gamma^{(k)}$.
Moreover, a sequence of flips in the neighborhood of $\gamma$
yields new tilings $t_0^{[k;\gamma]}$ and $t_1^{[k;\gamma]}$
such that $t_1^{[k;\gamma]} - t_0^{[k;\gamma]} = \gamma^{(k)}$.
If $t_\basetiling$ is the base tiling of $R$,
take $t_\basetiling'$ to be the base tiling of $R'$.
It then follows from the construction of $t_1^{[k;\gamma]}$
that $\Flux(t') = \Flux(t)$.
\end{proof}

\section{Surfaces}
\label{sec:surface}

In order to better understand the $\Flux$, and to define the Twist, we will work heavily with discrete surfaces.
Consider a cubiculated region $R$ and its dual $R^\ast$.
An  {\it embedded discrete surface} in $R^\ast$ is a pair
$(S,\psi)$ where:
\begin{itemize}
\item{ $S$ is an oriented topological surface
with (possibly empty) boundary $\partial S$;}
\item{ $\psi: S \to R^\ast \subset R$ is an injective continuous map
whose image is a union of vertices, edges and squares of $R^\ast$.}
\end{itemize}

We will sometimes  abuse notation using $S$ to refer to
 the domain surface $S$,
the image $\psi[S]$ and 
the element of $C_2(R^\ast;\ZZ)$ obtained by adding
the squares in $\psi[S]$
(with the orientation given by $S$ and $\psi$).

\begin{figure}[ht]
\centerline{\includegraphics[scale=0.25]{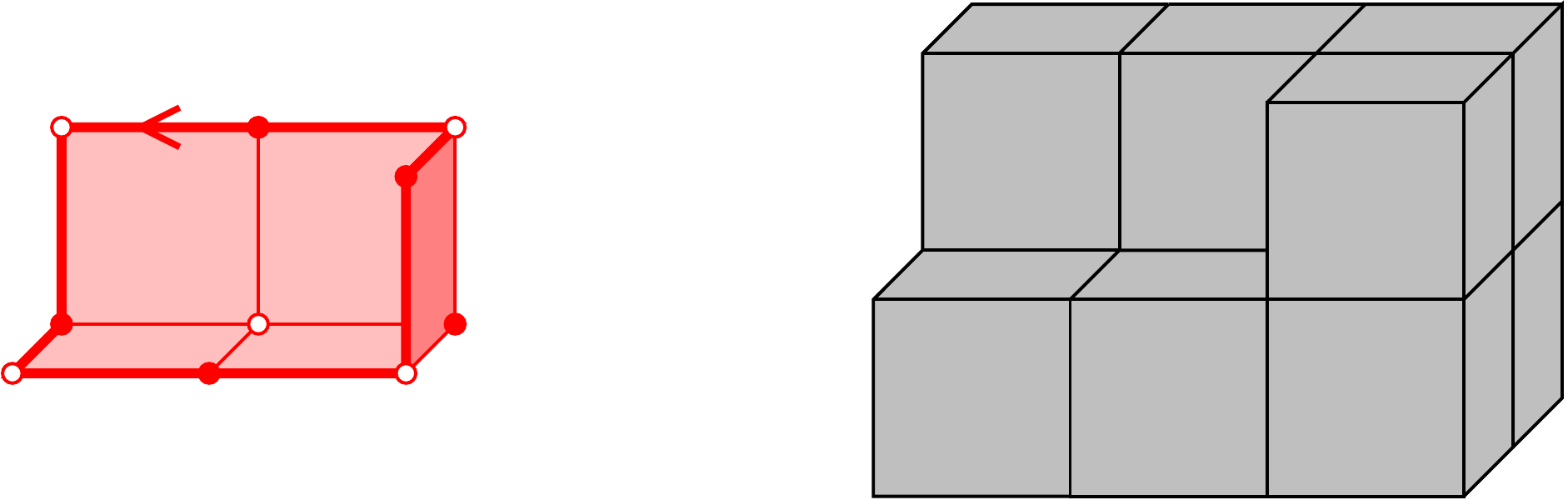}}
\caption{An embedded discrete surface which consists of five squares.
The left figure shows the surface in $R^\ast$,
the thicker line represents the oriented boundary of the surface.
The right figure shows the corresponding cubes in $R$.}%
\label{fig:surface_example}%
\end{figure}

Figure \ref{fig:surface_example} shows a simple embedded discrete surface.
Since both $S$ and $R$ are oriented,
this defines an orientation transversal to $S$.
In other words, at each square of $S$
there is a well defined normal vector.

Consider a tiling $t_0$ of a region $R$
and an embedded discrete surface $S$.
The tiling and the surface
are  {\it tangent at the boundary}
if $t_0$ and the system of cycles $\psi|_{\partial S}$ are tangent.
In particular, if $\partial S = \varnothing$ then
any tiling is tangent at the boundary to $S$.
The tiling and the surface are  {\it tangent}
if they are tangent at the boundary and, furthermore,
every vertex of $R^\ast$ in $S$
is an endpoint of a dimer in $t_0$ and $S$.

Refinements are important throughout the paper
and some additional remarks on the subject are in order.
Consider a tiling $t_0$ of $R$ tangent to $S$.
The refinement $t_0'$ of $t_0$ is not tangent to $S$ but,
as in Figure \ref{fig:refine},
a few flips are sufficient to go from $t_0'$ to a tiling $t_1$
which is tangent to $S$.
We will pay little attention to the distinction
between $t'_0$ and $t_1$ and speak of $t_1$ as the refinement of $t_0$.

We shall also have to consider tilings $t_0$ of $R$ which
are tangent to $\partial S$ but which cross $S$
(that is, fail to be tangent to $S$). 
If $t_0$ crosses $S$ at $v$ vertices, then
the refinement $t_0'$ crosses $S$ at many more points
(between $9v$ and $25v$).
Again, a few flips take $t_0'$ to $t_1$ which crosses $S$
at the original $v$ points only;
we often simply  call $t_1$ the refinement of $t_0$, slightly abusing notation.

\begin{definition}
A  \emph{ (discrete) Seifert surface} for a pair of tilings $(t_0,t_1)$
of a region $R$ is a connected embedded discrete (oriented) surface $S$ where
the restriction $\psi|_{\partial S}$ is the collection of nontrivial cycles of
$t_1 - t_0$ (respecting orientation).
\end{definition}

By definition, both $t_0$ and $t_1$ are tangent at the boundary
to $S$.
The tiling $t_0$ is tangent to $S$ if and only if $t_1$ is.
As above, if $S$ is a Seifert surface for $(t_0,t_1)$
then its refinement $S'$ is a Seifert surface for the pair
of refinements $(t_0',t_1')$.

\begin{example}
If $t_0$ and $t_1$ differ by a flip,
a unit square is a valid Seifert surface for the pair
(the simplest surface, but not the only one).
If $t_0$ is obtained from $t_1$ after a single positive trit,
we may assume that the situation is, perhaps after some rotation,
as portrayed in Figure \ref{fig:surfaceTritExample}.
Note that in order to build the surfaces portrayed
in Figure \ref{fig:surfaceTritExample},
we need that the interior point of the surface
in either case is a center of a cube in $R$. 
This condition is satisfied in at least one of the cases since  $\partial R$ is a manifold. 
\end{example}

\begin{figure}%
\centerline{\scalebox{.7}{\begin{tikzpicture}

\coordinate (V0) at (0,0);
\coordinate (V1) at (2,0);
\coordinate (V2) at (1,1);
\coordinate (V3) at (3,1);
\coordinate (V4) at (0,2);
\coordinate (V5) at (2,2);
\coordinate (V6) at (1,3);
\coordinate (V7) at (3,3);

\coordinate (w0) at (3,0);
\coordinate (w1) at (4,0);
\coordinate (w2) at (3.5,.5);
\coordinate (w3) at (4.5,.5);
\coordinate (w4) at (3,1);
\coordinate (w5) at (4,1);
\coordinate (w6) at (3.5,1.5);
\coordinate (w7) at (4.5,1.5);

\path[opacity=.7, fill=lightgray ] (V0) -- (V1) -- (V5) -- (V4) -- cycle;

\path[opacity=.7, fill=lightgray ] (V0) -- (V2) -- (V6) -- (V4) -- cycle;

\path[opacity=.7, fill=lightgray ] (V4) -- (V5) -- (V7) -- (V6) -- cycle;
  

%

\draw[->-=.5, line width=2pt] (V0) -- (V1);
  
\draw (V2) -- (V3);
\draw (V4) -- (V5);
\draw[->-=.5, line width=2pt] (V7) -- (V6);
\draw (V0) -- (V4);
\draw[->-=.5, line width=2pt] (V1) -- (V5);

\draw[->-=.5, line width=2pt] (V6) -- (V2);
\draw (V3) -- (V7);

\draw[->-=.5, line width=2pt] (V2) -- (V0);
\draw (V1) -- (V3);
\draw (V4) -- (V6);

\draw[->-=.5, line width=2pt] (V5) -- (V7);






\end{tikzpicture} \, \, \,  \begin{tikzpicture}

\coordinate (V0) at (0,0);
\coordinate (V1) at (2,0);
\coordinate (V2) at (1,1);
\coordinate (V3) at (3,1);
\coordinate (V4) at (0,2);
\coordinate (V5) at (2,2);
\coordinate (V6) at (1,3);
\coordinate (V7) at (3,3);

\coordinate (w0) at (3,0);
\coordinate (w1) at (4,0);
\coordinate (w2) at (3.5,.5);
\coordinate (w3) at (4.5,.5);
\coordinate (w4) at (3,1);
\coordinate (w5) at (4,1);
\coordinate (w6) at (3.5,1.5);
\coordinate (w7) at (4.5,1.5);

\path[opacity=.7, fill=lightgray ] (V0) -- (V1) -- (V3) -- (V2) -- cycle;
\path[opacity=.7, fill=lightgray ] (V2) -- (V3) -- (V7) -- (V6) -- cycle;
\path[opacity=.7, fill=lightgray ] (V1) -- (V3) -- (V7) -- (V5) -- cycle;
  

%

\draw[->-=.5, line width=2pt] (V0) -- (V1);
  
\draw (V2) -- (V3);
\draw (V4) -- (V5);
\draw[->-=.5, line width=2pt] (V7) -- (V6);
\draw (V0) -- (V4);
\draw[->-=.5, line width=2pt] (V1) -- (V5);

\draw[->-=.5, line width=2pt] (V6) -- (V2);
\draw (V3) -- (V7);

\draw[->-=.5, line width=2pt] (V2) -- (V0);
\draw (V1) -- (V3);
\draw (V4) -- (V6);

\draw[->-=.5, line width=2pt] (V5) -- (V7);






\end{tikzpicture}}}
\caption{Two possible Seifert surfaces.}
\label{fig:surfaceTritExample}%
\end{figure}

It follows from homology theory that
in order for a discrete Seifert surface to exist,
we must have $\Flux(t_0) = \Flux(t_1)$.
The converse is not true:
in the example of
Figure \ref{fig:233cyc} there exists 
a disconnected surface (two disks) 
and a connected surface which does not respect orientation
(a cylinder) but no discrete Seifert surface for the pair.
A smooth Seifert surface exists: see Figure \ref{fig:233seifert}.
This is one of several occasions when taking refinements
solves our difficulties.
Similarly, in Figure \ref{fig:442} the desired
smooth Seifert surface is a torus minus two disks:
there exists a discrete Seifert surface after taking refinements.

\begin{lemma}
\label{lemma:existseifert}
Consider a cubiculated region $R$
and two tilings $t_0$ and $t_1$ of $R$.
If $\Flux(t_0) = \Flux(t_1)$ then
for sufficiently large $k \in \NN$ there
exists a discrete Seifert surface in $R^{(k) \ast}$
for the pair $(t_0,t_1)$.
\end{lemma}

First, we need the following essentially topological lemma.
The statement is well-known for sufficiently nice regions,
see e.g.~\cite{Lickorish}. 
As we were unable to find a proof at our level of generality,
we include one here for completeness.  

\begin{lemma}\label{lemma:sublemma}
Consider a cubiculated region $R$ and two tilings $t_0$ and $t_1$ of $R$.
If $\Flux(t_0) = \Flux(t_1)$ then
there exists a smooth Seifert surface in $R^{(1) \ast}$
for the pair $(t_0, t_1)$.
\end{lemma}

\begin{figure}[ht]%
\centering
\centerline{\includegraphics[scale=0.25]{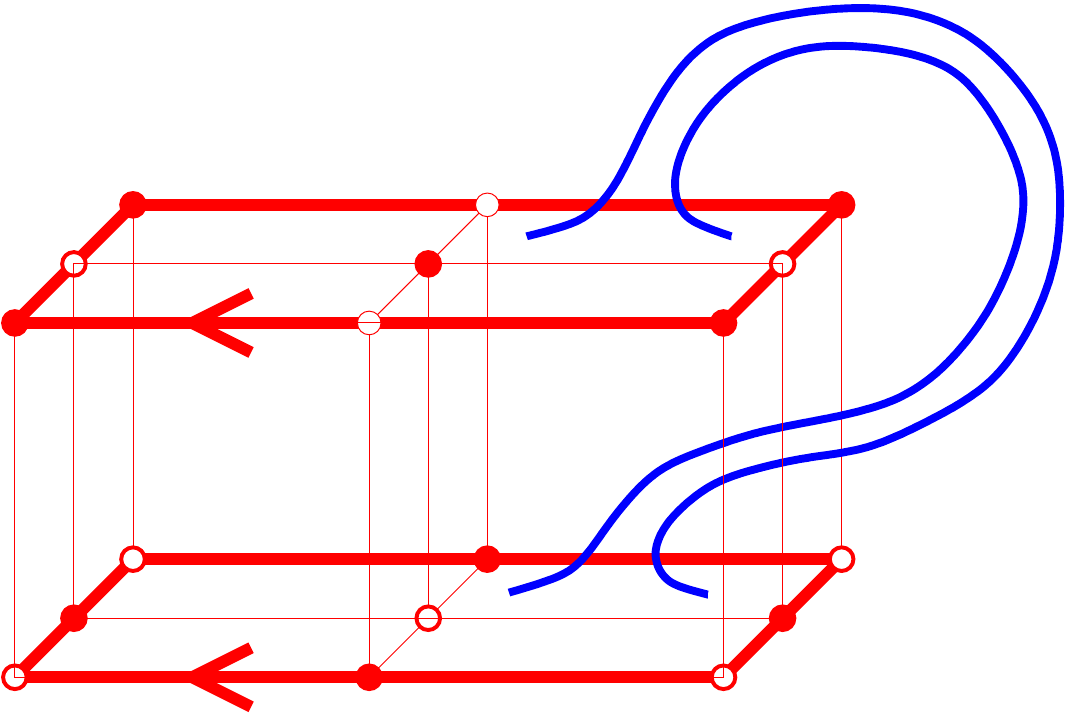}}
\caption{Let $t_0$ and $t_1$ be the two tilings
in Figures \ref{fig:233} and \ref{fig:233ast}.
Their difference is the union of two non trivial curves,
as in Figure \ref{fig:233cyc}.
There is no discrete Seifert surface,
but there exists a smooth Seifert surface.
There exists a discrete Seifert surface
in an appropriate refinement.}
\label{fig:233seifert}
\end{figure}

\begin{proof}
To simplify notation, write $L$ for the difference $t_1 - t_0$.
The refinement guarantees that $L$ is contained in the interior of $R$.
The hypothesis $\Flux(t_0) = \Flux(t_1)$ guarantees that $L$ is a boundary,
i.e.,
that there exists $s$ in $C_2(R^{(1) \ast}) $
with $\partial(s) = L$.

For each vertex $v$ of $R$, construct a small open ball $b_v$ around
$v$.  For each edge $e$ of $R$, construct a thin open cylinder $c_e$
around $e$, where the radii of the cylinders is much less than that of
the balls.  Let $R_0 = R \, \smallsetminus \, \{b_v \cup c_e \}$, $R$
minus the union of all $b_v$ and $c_e$.  Let $R_1 = R \,
\smallsetminus \, \{b_v\}$, $R$ minus the union of all $b_v$.  Thus
$R_0 \subset R_1 \subset R$.  We construct a smooth Seifert surface in
three stages: first in $R_0$, then extend it to $R_1$ and finally to
$R$.

For $R_0$, we consider each square $a$ of $R$ and its coefficient $s_a$ in $s$.
Orient the square $a$ so that $s_a \geq 0$.
Construct $S_0$ in $R_0$ by taking $s_a$ translated copies of $a$,
with boundary falling outside $R_0$.

Next consider each edge $e$. There are two possibilities:
$e$ may or may not belong to the support of $L$.
First assume it does not.
Examine the boundary of $c_e$: we see a number of line segments
(the intersection of the squares in $S_0$ with the boundary of $c_e$).
This can be described by a family of $2k$ points in a circle,
$k$ positive and $k$ negative.
It is possible to match positive and negative points
and draw curves joining them
so that the curves do not cross
(see Figure \ref{fig:circles}).
Indeed, by induction, take two adjacent points,
one positive and one negative,
and join them by a curve near the circle.
For the other points construct segments taking to a smaller circle.
Now use induction on $k$.
Take the Cartesian product of these curves by $e$
to construct a surface in $R_1$.

\begin{figure}[ht]%
\centering
\centerline{\includegraphics[scale=0.25]{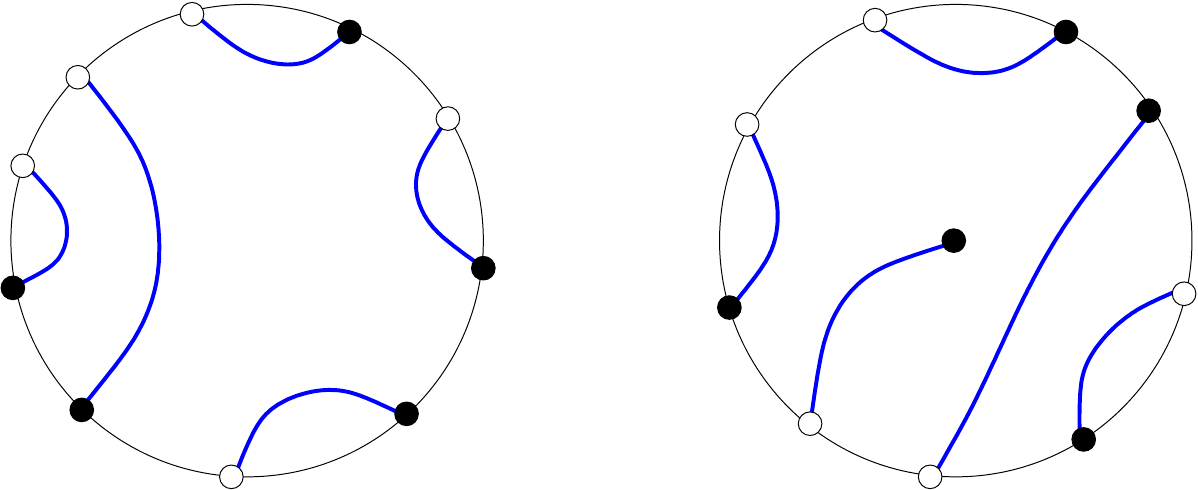}}
\caption{If there are $k$ black and $k$ white points 
on a circle then they can be joined by disjoint arcs.
Similarly, if there are $k$ black and $k+1$ white points 
on a circle and the center is also black,
then they can also be joined by disjoint arcs.}
\label{fig:circles}
\end{figure}

Now consider the case where $e$ belongs to the support of $L$.
Similar to above, we have $k$ positive points on a circle,
$(k+1)$ negative points on the circle
and a positive origin (see again Figure \ref{fig:circles}).
The same inductive proof constructs smooth disjoint curves joining the points
(the last negative point on the circle will be connected to the center).
Again, take the Cartesian product of these curves with $e$
to construct a surface in $R_1$.
This completes the construction of the surface $S_1$ in $R_1$.

We now extend the surface to $R$, that is, we extend it to each ball $b_v$.
Again we consider two cases:
$v$ in the support of $L$ and $v$ not in support of $L$.
First consider $v$ not in the support of $L$.
Examine the boundary of $b_v$, a sphere $S_v$.
Our previous construction obtains
a family of disjoint oriented simple closed curves in $S_v$.
Attach disjoint disks contained in $b_v$ with
these curves as
boundary:
Take a point in $S_v$ not in the cycles and call it infinity,
so that $S_v$ is identified with the plane.
Cycles are now nested: start with an innermost cycle and close it;
proceed along cycles to the outermost.

Next, consider the case where $v$ is in the support of $L$.
Again examine the sphere $S_v$.
We have a family of disjoint oriented simple curves:
one of them is a segment (with two endpoints), the others are closed (cycles).
Take the ``point at infinity" in $S_v$ very near the segment
so that in the plane the segment is outside the cycles.
Close cycles from inner to outermost.

Finally, consider the segment.
Notice it can be long, perhaps going several times
around a face of the octahedron
formed on the sphere by the adjacent cubes.
Take a smooth $1$-parameter family of diffeomorphisms
keeping the endpoints of the segment fixed and
taking the segment to a geodesic on the sphere
(this may involve rotations around endpoints).
Apply this family of diffeomorphisms on spheres of decreasing radii
and complete the surface with a plane near the vertex.
This completes the construction of $S$.
\end{proof}

\begin{proof}[Proof of Lemma~\ref{lemma:existseifert}]
Construct a smooth Seifert surface $\psi_{\infty}: S \to R$
(as in Lemma~\ref{lemma:sublemma}).
A sufficiently large value of $k$
allows for an approximation $\psi$ of $\psi_{\infty}$
such that $\psi$ is a discrete Seifert surface as follows:

Consider a smooth Seifert surface $S$ as in Lemma~\ref{lemma:sublemma}.
Let $K$ be the maximum sectional curvatures of $S$,
up to and including the boundary.  
Take $n$ such that $K < 5^{(n-1)}$ and refine $n$ times
so that the radii of curvature at any point is always 
more than twice the diagonal of any cube.
Now classify cubes near the surface
(meaning with center at a distance $< 5/2$ from the surface,
thus including cubes crossing the surface)
as \emph{above} or \emph{below} $S$ according to the orientation of $S$
and the measure of the cube on each side of $S$.
The tiled surface $S$ lies between cubes
which are \emph{above} and cubes \emph{below} 
(and therefore closely approximates $S$).
The curvature estimate implies the good behavior of $S$.
\end{proof}

\section{Flux through surfaces}
\label{sec:fluxthroughsurfaces}

We now introduce \emph{flux through surfaces}.
We relate this notion with the $\Flux(t)$ at the end of the section.
The interpretation here provides motivation for our choice of terminology;
we think of tilings as flowing through regions and across surfaces.

Let $t$ be a tiling and $S$ an embedded discrete surface such that
$t$ is tangent to $S$ at the boundary.
For each vertex $v$ of $R^\ast$ in the interior of $S$
consider the only dimer $d$ of $t$ adjacent to $v$.
Draw a small vector along $d$ starting from $v$:
its endpoint may be on $S$ (if $d$ is contained in $S$),
above, or below $S$.

When we speak of above and below $S$, this is to be understood as follows.
The ambient space $R$ is an oriented 3D manifold;
the surface $S \subset R$ is also oriented
thus giving us a canonical normal vector, at least on squares.
This normal vector is understood to be pointing up.

\begin{definition}
\label{definition:fluxthroughS}
For a surface $S$ and a tiling $t$,  define the \emph{flux of $t$ through $S$}, $\phi(t; S) \in \ZZ$, as follows. Set
\[ \varphi(v;t;S) = \ccolor(v) \cdot
\begin{cases}
+1, &\text{ endpoint above } \psi[S]; \\
0, &\text{ endpoint on } \psi[S]; \\
-1, &\text{ endpoint below } \psi[S];
\end{cases} \qquad
\phi(t;S) = \sum_v \varphi(v;t;S), \]
where the color of $v$ is $+1$
if $v$ corresponds to a black vertex (or cube)
and $-1$ if $v$ corresponds to a white vertex. 
\end{definition}

\begin{example}
\label{example:surface_simple}
In Figure \ref{fig:surface_simple},
we see four dimers intersecting the interior of a surface. 
The horizontal dimer completely contained
in the surface will contribute $0$ to the flux.  
\end{example} 

\begin{figure}[ht]%
\centering
\centerline{\includegraphics[scale=0.25]{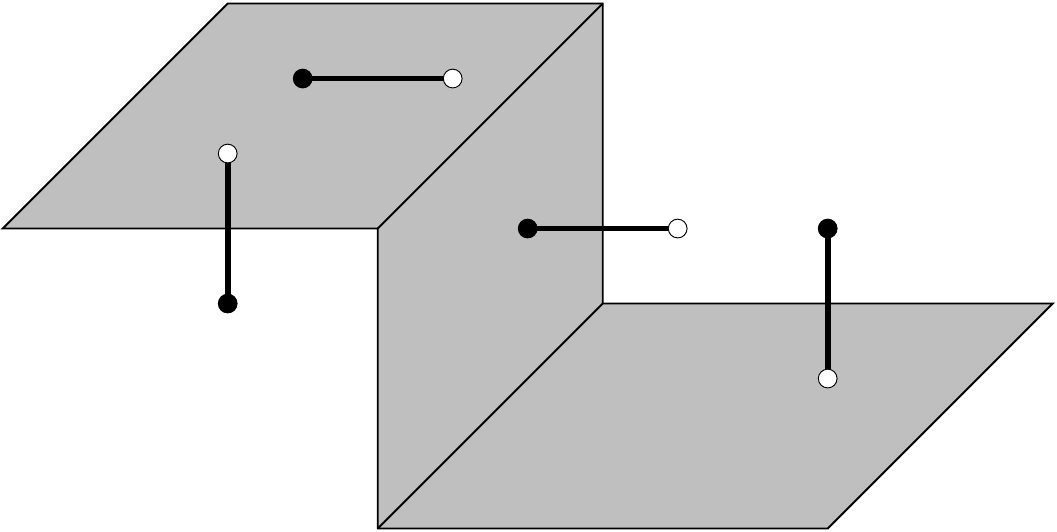}}
\caption{An example of flow through a surface. 
The surface is drawn in $R^\ast$.
Only a few representative dimers have been shown;
one sitting strictly in the interior of the surface,
two above and one below.}
\label{fig:surface_simple}%
\end{figure}

\begin{example}
\label{example:phi}
The first row of Figure \ref{fig:phi} shows two tilings $t_0$ and $t_1$
of the $4\times 4\times 4$ cube.
The difference $t_1 - t_0$ is a cycle of length $12$,
also indicated on the first row.
In the second row we consider two different surfaces
$S_0$ and $S_1$ with $\partial S_0 = \partial S_1 = t_1 - t_0$
(there are many others).
We then compute $\varphi(\cdot)$ for each vertex (cube) in each surface.
In the figure, white bullets indicate $\varphi(\cdot) = -1$
and black bullets indicate $\varphi(\cdot) = +1$.
Notice that in all $4$ cases we have $\phi(t_i,S_j) = -1$:
we shall soon prove that this is not a coincidence.
\end{example}

\begin{figure}[ht]%
\centerline{\includegraphics[scale=0.25]{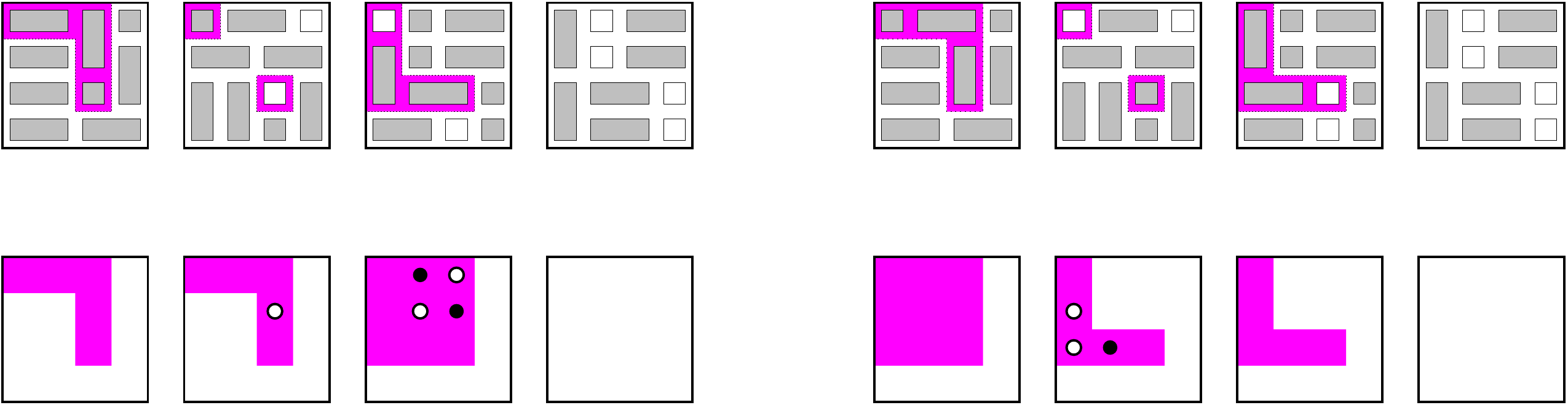}}
\caption{Two tilings $t_0$ and $t_1$ of the $4\times 4\times 4$ cube;
the difference $t_1 - t_0$ is a cycle.
The highlighted cubes are two surfaces $S_0$ and $S_1$ with
$\partial S_0 = \partial S_1 = t_1 - t_0$, drawn in $R$ as sets of cubes.
The bullets indicate contributions $\varphi(\cdot)$
and the color indicates the sign.
We have $\phi(t_i,S_j) = -1$ in all $4$ cases.}
\label{fig:phi}%
\end{figure}

\begin{example}
\label{example:fluxthrough666}
Consider the $6\times 6\times 6$ torus
and the tilings $t_0$ and $t_1$ in Figure \ref{fig:666f1}.

The third floor (i.e., third $6\times 6$ square)
in each tiling is a 2D torus $S_1$,
an embedded discrete surface.
There are $36$ dominoes with one vertex on the torus $S_1$
and the second one above (i.e., on the next floor).
There are $18$ black and $18$ white cubes and therefore
$\phi(t_0;S_1) = \phi(t_1;S_1) = 0$.
Notice that if $\tilde S_1$ is another floor we have
$\phi(t_i;S_1) = \phi(t_i;\tilde S_1)$.

The sixth column (of each floor)
is a 2D torus $S_2$, another embedded discrete surface.
All dominoes are contained in $S_2$ and therefore we have
$\phi(t_0;S_2) = \phi(t_1;S_2) = 0$.

Finally, the sixth row (of each floor)
is another embedded discrete surface $S_3$.
Now there are two dominoes with endpoints not in $S_3$,
one above, one below.
The vertices $v$ for each domino have different colors.
We therefore have
$\phi(t_0;S_3) = \phi(t_1;S_3) = \pm 2$,
with sign depending on the orientation of $S_3$.

The tori $S_1, S_2, S_3$ generate $H_2(R)$.
It follows from Lemma \ref{lemma:homology} below
that we thus know $\phi(t_i;S)$ for any
closed embedded surface $S$
provided we know how to write $[S] \in H_2$
as a linear combination of $[S_1], [S_2], [S_3] \in H_2$.

On the other hand, for the four tilings in Figure \ref{fig:666f0}
we have $\Flux(\cdot) = 0$ and $\phi(\cdot;S) = 0$
for any closed embedded discrete surface $S$.
\end{example}

The next Theorem shows that with the definition above,
the flux of a tiling through the boundary of a manifold is always zero.  

\begin{theo}
\label{thm:euler}
Let $R$ be a cubiculated region and $t$ a tiling of $R$.
Let $\psi: S \to R^\ast$ be an embedded discrete surface
with $\partial S = \varnothing$.
Assume there exists a topological manifold $M_1 \subset R^\ast$
with ${\partial} M_1 = S$. Then, $ \phi(t;S) = 0.$

 \end{theo}

\begin{proof}

We start by showing the following enumerative result.  
Let $b_{\interior}$ and $w_{\interior}$ be
the number of black and white vertices
of $R^\ast$ in the interior of $M_1$;
let $b_{\partial}$ and $w_{\partial}$ be
the number of black and white vertices of $R^\ast$ on $S$. 
Then $2b_{\interior} + b_{\partial} = 2w_{\interior} + w_{\partial}$.

To prove this claim, let $M_2$ be a copy of $M_1$ with reversed orientation.
Glue $M_1$ and $M_2$ along $S$ to define a topological $3$-manifold $M$.
The manifold $M$ is oriented, has empty boundary,
and inherits from $M_1$ and $M_2$ a cell decomposition,
with vertices painted black and white.
Let $f_i(M)$ be the number of faces of dimension $i$
in this cell decomposition.
Since the Euler characteristic of $M$ equals $0$,
\[ f_0(M) - f_1(M) + f_2(M) - f_3(M) = 0. \]
Also, $f_2(M) = 3 f_3(M)$:
since $M$ has no boundary, each face must be shared by exactly two cubes
and each cube has six faces.

Our enumerative claim is equivalent to saying that the number of black vertices
of $M$ equals the number of white vertices.
In order to see this, we first build another $3$-complex $T$ for
the topological manifold $M$ in the following way:
the vertices of $T$ are the white vertices of $M$;
the edges of $T$ are the diagonals that connect white vertices in each
(two-dimensional) face of $M$;
the two-dimensional faces of $T$ are the four triangles
in each cube that form a regular tetrahedron with its four white vertices;
finally, the three-dimensional faces of $T$ come in two flavors:
the regular tetrahedrons inside each cube,
and cells around black vertices of $M$
(these are regular octahedra if the vertex is in the interior
of either $M_i$ but may have some other shape if the vertex belongs to $S$).

We have $f_0(T) = 2w_{\interior} + w_{\partial}$,
$f_1(T) = f_2(M)$ (each face of $M$ contains exactly one edge of $T$),
$f_2(T) = 4 f_3(M)$ (each cube of $M$ contains exactly four faces of $T$)
and
$f_3(T) = f_3(M) + 2b_{\interior} + b_{\partial}$ 
(one tetrahedron inside each cube, one other cell around each black vertex).
Thus
\begin{align*}
0 &= f_0(T) - f_1(T) + f_2(T) - f_3(T) \\
&= (2w_{\interior} + w_{\partial} - 2b_{\interior} - b_{\partial}) +
(3 f_3(M) - f_2(M))
= 2w_{\interior} + w_{\partial} - 2b_{\interior} - b_{\partial}.
\end{align*}

Now, each black vertex on $S$ must match either
a white vertex in the interior of $M_1$, on $\partial M_1 = S$,
or in the exterior of $M_1$:
let $b_1$, $b_2$ and $b_3$ be the number of vertices of each kind
so that $b_1 + b_2 + b_3 = b_{\partial}$.
Define $w_1$, $w_2$ and $w_3$ similarly,
so that $w_1 + w_2 + w_3 = w_{\partial}$.
Clearly, $b_2 = w_2$.
By definition, $\phi(t;S) = b_3 - b_1 - w_3 + w_1$.
Counting vertices in the interior of $M_1$ gives
$b_1 - w_1 = w_{\interior} - b_{\interior}$
so that
$\phi(t;S) = (b_{\partial} - w_{\partial}) - 2(b_1 - w_1) = 0$.
\end{proof}

The remainder of the section develops
that the flux is not dependent on a precise surface,
but can be defined in terms of homology classes.  
If $S_0$ and $S_1$ are oriented surfaces
with $\partial S_0 = \partial S_1$ 
then let $S_1 - S_0$ denote the surface obtained by gluing $S_1$ and $S_0$
along the boundary
and reverting the orientation of $S_0$.
Furthermore, if $\psi_i: S_i \to R$ are continuous maps,
let $\psi_1 - \psi_0: S_1 - S_0 \to R$
denote the map defined by $(\psi_1 - \psi_0)(p) = \psi_i(p)$
if $p \in S_i$.

The image, 
$(\psi_1 - \psi_0)[S_1 - S_0]$, can be seen as  an element of $H_2(R)$; 
the maps $\psi_0$ and $\psi_1$ are  {\it homological}
if $(\psi_1 - \psi_0)[S_1 - S_0] = 0 \in H_2(R)$.
Note that if $S_0$ and $S_1$ are smooth or topological oriented surfaces
with $\partial S_0 = \partial S_1$
and $\psi_i: S_i \to R$ are smooth or topological embeddings
such that there exists a topological manifold $M_1 \subset R$
for which $\psi_1 - \psi_0: S_1 - S_0 \to \partial M_1 \subset R$
is an orientation preserving homeomorphism
then $\psi_0$ and $\psi_1$ are homological.
On the other hand,
if $S_0$ and $S_1$ are smooth or topological oriented surfaces
with $\partial S_0 = \partial S_1$
then the map $\psi_1 - \psi_0: S_1 - S_0 \to R$
is usually not an embedded surface.

\begin{lemma}
\label{lemma:homology}
Let $R$ be a cubiculated region and $t$ be a tiling of $R$.
Let $S_0$ and $S_1$ be oriented surfaces with $\partial S_0 = \partial S_1$.
Let $\psi_i: S_i \to R^\ast$ be embedded discrete surfaces
with $(\psi_0)|_{\partial S_0} = (\psi_1)|_{\partial S_1}$.
If $t$ is tangent to $(\psi_i)|_{\partial S_i}$
and $\psi_0$ and $\psi_1$ are homological 
then $\phi(t,S_0) = \phi(t,S_1)$.
\end{lemma}

\begin{example}
\label{example:homologyphi}
The situation in Lemma \ref{lemma:homology}
is illustrated in Example \ref{example:phi} and in Figure \ref{fig:phi}.
Indeed, the surfaces $S_0$ and $S_1$ satisfy $\partial S_0 = \partial S_1$.
Notice that $\phi(t_i,S_0) = \phi(t_i,S_1)$ for $i = 0$ and $i = 1$.
\end{example}

We construct a function $\omega$, the winding number,
taking integer values on vertices of $R$
(and therefore also cubes of $R^\ast$).
Consider $v_0$, $v_1$ vertices of $R$:
we first show how to compute $\omega(v_1) - \omega(v_0)$.
Consider a simple path $\gamma$ along edges of $R$ going from $v_0$ to $v_1$
(such a path exists since $R$ is assumed to be connected).
Count intersections of $\gamma$ with (the image of) $\psi_1 - \psi_0$.
Notice that $\gamma$ intersects $\psi_1 - \psi_0$
at the centers of oriented squares:
each intersection counts as $+1$ (resp. $-1$)
if the tangent vector to $\gamma$ coincides (resp. or not)
with the normal vector to the square in $\psi_1 - \psi_0$.
This total is $\omega(v_1) - \omega(v_0)$.

Notice that the value of $\omega(v_1) - \omega(v_0)$ does not depend
on the choice of the path $\gamma$.
Indeed, take two such paths $\gamma_0$ and $\gamma_1$
and concatenate them to obtain a closed path $\gamma_1 - \gamma_0$.
Counting intersections with $\gamma_1 - \gamma_0$
as described in the previous paragraph
defined a linear map from $C_2(R^\ast)$ to $\ZZ$
and therefore an element of $C^2(R^\ast)$
which is easily seen to be in $Z^2(R^\ast)$.
Since $\psi_1 - \psi_0$ is assumed to belong to $B_2(R^\ast)$
their product must equal $0$,
yielding independence from the path.
A similar argument shows that if $v_0$ and $v_1$
both belong to the boundary $\partial R$
then counting intersections with $\gamma$ also defines
an element of $Z^2(R^\ast)$ and therefore $\omega(v_1) - \omega(v_0) = 0$.
We may therefore define $w$ so that if $v \in \partial R$ then $\omega(v) = 0$;
if $\partial R = \varnothing$ we have a degree of freedom here
and we choose $w$ so that it assumes the value $0$ somewhere.

\begin{proof}[Proof of Lemma~\ref{lemma:homology}]
Our proof works by induction of $c = |\max \omega| + |\min \omega|$.
If $c = 0$ then the surfaces $\psi_0$ and $\psi_1$
coincide and we are done.
Let us consider the case $c = 1$;
without loss of generality, $\omega$ assumes the values $0$ and $1$.
Let $M_1 \subset R^\ast$ be the union of closed cubes
centered at vertices $v$ (of $R$) with $\omega(v) = 1$.
We would like $M_1$ to be a 3D manifold with boundary.
Unfortunately, that is not guaranteed.
But this can easily be fixed.
Start by refining $R$ (and the surfaces),
so that now $M_1$ is a union of $5\times 5\times 5$ blocks of cubes.
If $M_1$ is \emph{not} a manifold, this means there are
bad edges (two alternate blocks present, two absent) or
bad vertices (more than one undesirable pattern).
First fix the vertices by adding extra cubes; then the edges.
This corresponds to constructing a chain of auxiliary surfaces
\[ \tilde\psi_0 = \psi_0, \tilde\psi_1, \tilde\psi_2, \ldots,
\tilde\psi_{N-1}, \tilde\psi_N = \psi_1 \]
such that for any $k$ the pair $\tilde\psi_{k-1}$, $\tilde\psi_k$
satisfies $c = 1$ and $M_1$ a manifold.

\begin{figure}[h]
\centerline{\includegraphics[width=2in]{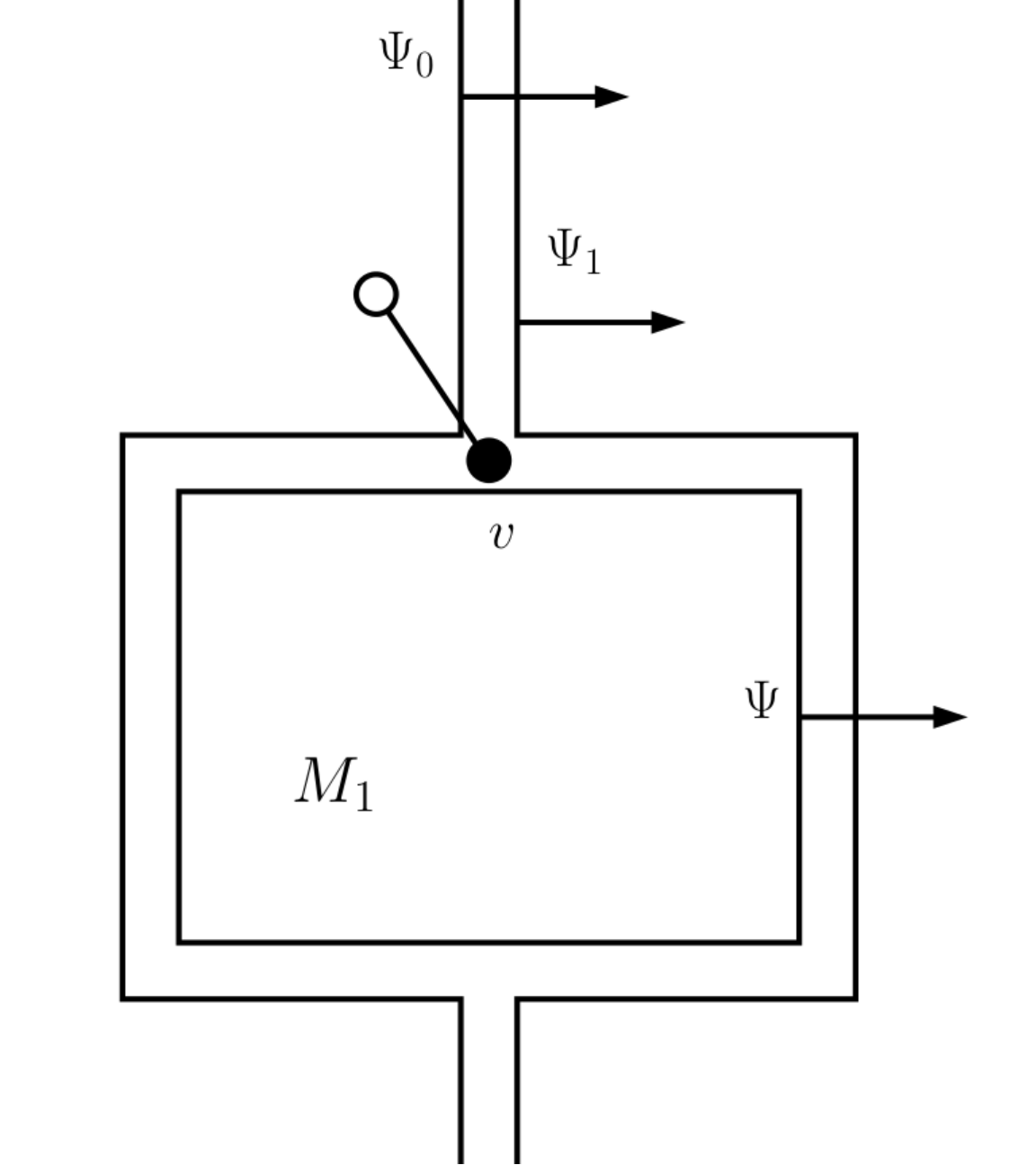}}
\caption{Auxiliary surfaces as in the proof of Lemma~\ref{lemma:homology}.}
\label{fig:X} 
  \end{figure}

We are therefore left with the case where
$\psi_1 - \psi_0$ is the boundary of a manifold $M_1$.
It now appears that this case follows from Theorem \ref{thm:euler}.
This is true but not as trivial as it may seem at first.
As in Figure~\ref{fig:X}, let $S = \partial M_1$
and $\psi$ be a parameterization of $S$.
What Theorem \ref{thm:euler} tells us is that $\phi(t;S) = 0$ but
what we need is that $\phi(t;S_1) - \phi(t;S_0) = 0$.
In other words, we need
\begin{equation}
\label{eq:Dv}
\sum_v D(v) = 0; \qquad
D(v) := \varphi(v;t;S_1) - \varphi(v;t;S_0) - \varphi(v;t;S). 
\end{equation}
If $v$ belongs to at most two surfaces then $D(v) = 0$.
On the other hand, for $v$ as in Figure~\ref{fig:X}, $D(v) = -1$.
Let $\Gamma$ be the curve along which the three surfaces
$S_0$, $S_1$ and $S$ meet.
A case by case analysis shows that for $v \in \Gamma$
we have $D(v) = 0$ if $v$'s partner also belongs to $\Gamma$
and $D(v) = -\ccolor(v)$ otherwise.
Since $\Gamma$ is balanced this proves equation \ref{eq:Dv}
and completes the proof of this case.

The general inductive step is now similar.
Otherwise assume without loss of generality that $\max \omega = l > 0$.
Let $M_l$ be the union of cubes of $R^\ast$
with center $v$ with $\omega(v) = l$.
As above, we may assume $M_l$ to be a 3D manifold.
Thus, $M_l$ is a 3D manifold and its boundary $\partial M_l$
consists of subsets of $\psi_0$ and $\psi_1$ meeting at curves
(in the simplest example, $M_l$ is a ball,
the two subsets are disks meeting at a circle;
this may get significantly more complicated but does not affect our argument).
Modify $\psi_0$ to define $\psi_2$ by discarding
the subset of $\psi_0$ in $\partial M_l$ and
attaching instead the subset of $\psi_1$ also in $\partial M_l$.
The two surfaces $\psi_0$ and $\psi_2$
are homological by construction and
the difference $\psi_2 - \psi_0$ is the boundary of $M_1$.
The case $c = 1$ above shows that $\phi(t,S_0) = \phi(t,S_2)$.
For the new winding number $\tilde \omega$
(defined for the pair $\psi_1$ and $\psi_2$)
we have $\tilde c = c - 1$:
by induction we are done.
\end{proof}

\begin{lemma}
\label{lemma:h2}
Consider a cubiculated region $R$, a tiling $t$ of $R$, and 
 an element $a \in H_2(R;\ZZ)$.
\begin{enumerate}
\item{There exist a nonnegative integer $k$ and an embedded discrete surface
$\psi: S \to R^{(k)\ast}$ such that $\partial S = \varnothing$
and $\psi[S] = a$.}
\item{Let $k_0$ and $k_1$ be nonnegative integers.
Let $\psi_0: S_0 \to R^{(k_0) \ast}$
and $\psi_1: S_1 \to R^{(k_1) \ast}$
be embedded discrete surfaces such that
$\partial S_0 = \partial S_1 = \varnothing$.
Assume that $\psi_0[S] = \psi_1[S] = a$.
Then $\phi(t^{(k_0)};S_0) = \phi(t^{(k_1)};S_1)$. }
\end{enumerate}
\end{lemma}

\begin{proof}
As in the proof of Lemma~\ref{lemma:existseifert}, for any $a \in H_2(R;\ZZ)$, there exists a smooth embedded surface
$\psi_{\infty}: S \to R$ with $\psi_{\infty}[S] = a$
(where $S$ is a smooth closed surface with $\partial S = \varnothing$).
For sufficiently large $k$ there is an embedded discrete surface
$\psi: S \to R$ approximating $\psi_{\infty}$
so that $\psi[S] = a$. 
The equality in item (2) follows from Lemma \ref{lemma:homology}.
\end{proof}

Using this lemma, for $a \in H_2(R;\ZZ)$, define $\phi(t;a)$ to be equal to
$\phi(t^{(k)};S)$ for any embedded discrete surface
$\psi: S \to (R^{(k)})^\ast$ such that $\psi[S] = a$ (as an element of $H_2(R;\mathbb{Z}))$.

Note that if $t_0$ and $t_1$ differ by a flip or trit
then 
$\phi(t_0;a) = \phi(t_1;a)$.

\begin{lemma}
\label{lemma:phi}
Consider a cubiculated region $R$ and tilings $t_0, t_1$ of $R$.
If $\Flux(t_0) = \Flux(t_1)$
then $\phi(t_0;a) = \phi(t_1;a)$ for all $a \in H_2(R;\ZZ)$.
\end{lemma}

\begin{proof}
Let $S$ be a discrete embedded surface in some refinement $R^{(k)}$
with $\partial S = \varnothing$.
Adding $\phi(t_0;S) + \phi(t_1;S)$ 
gives a linear map from
$C_1(R^\ast;\ZZ)$ to $\ZZ$.
Boundaries of squares are taken to $0$ and therefore
$B_1(R^\ast;\ZZ)$ is contained in the kernel of this map.
Hence we have a map from $H_1(R^\ast;\ZZ)$ to $\ZZ$.
By hypothesis $[t_1 - t_0] = 0 \in H_1$
and therefore $t_1$ and $t_0$ are taken to the same number.
In other words, $\phi(t_0;S) = \phi(t_1;S)$.
Since this holds for all $S$, $\phi(t_0;a) = \phi(t_1;a)$.
\end{proof}

The converse does not hold.
For example, let $\cL$ be spanned by the vector $(0,0,4)$;
let $R \subset (\RR^3 / \cL)$ be the set of points $(x,y,z)$
for which $0 \le x, y \le 4$.
Then, one can construct tilings of $R$ with different values of the Flux
but $H_2(R) = 0$ and therefore $\phi(t;a)$ is always trivial.

Most importantly, it follows from Lemma~\ref{lemma:phi} that the next
definition is sound and $m$ can be seen as a function of the
$\Flux(t)$ rather than $t$.

\begin{definition}
\label{definition:modulus}
Define the \emph{modulus}  of a tiling as
\[ m = \mu(\Flux(t)) = \gcd_{a \in H_2} \phi(t;a); \]
so that for all $a \in H_2$ we have $\phi(t;a) \equiv 0 \pmod {m}$.
\end{definition}

\begin{example}
\label{example:modulus}
Let $R$ be the $6\times 6 \times 6$ torus.
The four tilings in Figure \ref{fig:666f0} have Flux $0$
and for them $m = 0$.

The two tilings in Figure \ref{fig:666f1} have Flux $(1,0,0)$.
As we saw in Example \ref{example:fluxthrough666},
the torus $S_i$ ($i \in \{1,2,3\}$)
corresponds to $a_i \in H_2$;
the elements $a_1, a_2, a_3$ generate $H_2$.
We have $\phi(t_i;a_1) = \phi(t_i;a_2) = 0$ and
$\phi(t_i;a_3) = 2$ and therefore $m = 2$.
\end{example}

\section{Twist}
\label{sec:twist}

In this section, we define our second topological invariant of tilings,
the twist.  The twist was first introduced in~\cite{segundoartigo}. 
There it has a simple combinatorial definition but the
construction involved is not well defined
at the level of generality of  this paper.
Here we give an alternate  definition of twist involving
embedded surfaces. 
If the Flux is zero, the twist assumes values in $\ZZ$. 
Otherwise, the twist assumes values in $\ZZ/m\ZZ$
where $m$ is the modulus of the tiling, 
as defined at the end of the previous section.
Intuitively,
the twist records how ``twisted'' a tiling is by trits;
the value of the twist changes by exactly $1$ after a trit move.

We start by defining the \emph{flux around a curve}.    
Consider a cubiculated region $R$,
a tiling $t$ of $R$, and $m = \mu(\Flux(t))$.
If $\gamma: \Ss^1 \to R^\ast$ is an embedded cycle 
such that $t$ is tangent to $\gamma$,
then Lemmas~\ref{lemma:existseifert} and~\ref{lemma:phi}  imply that 
there exists a nonnegative integer $k$ and an embedded surface
$\psi: S \to R^{(k) \ast}$ such that
$\psi|_{\partial S} = \gamma^{(k)}$. 
Furthermore, if $k_0$ and $k_1$ are nonnegative integers
and $\psi_0: S_0 \to R^{(k_0) \ast}$
and $\psi_1: S_1 \to R^{(k_1) \ast}$
are embedded surfaces such that
$(\psi_i)|_{\partial S_i} = \gamma^{(k_i)}$, 
then
$\phi(t^{(k_0;\gamma)};S_0) = \phi(t^{(k_1;\gamma)};S_1)$
(as elements of $\mathbb{Z}/m\ZZ$).
These observations allow us to
define the flux of a tiling $t$ around a curve $\gamma$.
Notice that in the definition below
we use the modified refinement $t^{(k;\gamma)}$,
as in Lemma \ref{lem:refineflux} and Figure \ref{fig:refine}.

\begin{definition}
\label{definition:fluxaroundcurve}
For a tiling $t$ and a curve $\gamma$, define $\phi(t;\gamma) \in \ZZ/m\ZZ$,
the \emph{flux of $t$ around $\gamma$}, to be
$$  \phi(t;\gamma) :=   \phi(t^{(k;\gamma)};S) \in \ZZ/m\ZZ$$
for any surface $\psi: S \to R^{(k) \ast}$
such that $\psi|_{\partial S} = \gamma^{(k)}$.
\end{definition}

\goodbreak

Using the flux of a tiling around a curve, we may define our first
notion of twist; the twist for a pair of tilings.

\begin{definition}
Let $R$ be a cubiculated region.
Let tilings $t_0$ and $t_1$ be two tilings of $R$
such that $\Flux(t_0) = \Flux(t_1)$. 
Then the twist of $t_1$ with respect to $t_0$  is defined as
\begin{equation*}
\label{eq:dtwist}
\TW(t_1;t_0) := \phi(t_1;t_1-t_0) = \phi(t_0;t_1-t_0)  \in \ZZ/m\ZZ. 
\end{equation*}
\end{definition}

\begin{example}
\label{example:TWphi}
In Example \ref{example:phi} we have $\phi(t_1,S_0) = -1$
and therefore $\TW(t_1;t_0) = \phi(t_1;t_1-t_0) = \phi(t_1,S_0) = -1$.
\end{example}

Our larger goal is to define the twist of a single tiling, $\Tw(t)$.
In particular, the twist should satisfy
$\TW(t_1;t_0) = \Tw(t_1) - \Tw(t_0)$  
so that 
$\Tw(t)$ can be defined using a base tiling and the twist of a pair.
To this end, first consider the result of a flip move.
\begin{prop}
\label{prop:fliptwist}
Let $R$ be a cubiculated region.
Let $t_0$, $t_1$ and $t_2$ be tilings of $R$
such that $\Flux(t_0) = \Flux(t_1)$  and
$t_1$ and $t_2$ differ by a flip. Then
\[ \TW(t_1;t_2) = 0 \qquad and \qquad \TW(t_1;t_0) = \TW(t_2;t_0). \]
\end{prop}

\begin{proof}
For the first equation simply take the surface for $t_1 - t_2$ 
to be a single square.

For the second equation, first consider the case where the flip
$t_2-t_1$ is disjoint from the system of cycles $t_1-t_0$.
Then we may take a surface $S_1$ for $t_1 - t_0$
which is also disjoint from the flip $t_2-t_1$.
Take $S_2$ to be the disjoint union of $S_1$
with the square with boundary $t_2 - t_1$.
For the general case, we have to consider the position
of the single square with respect to the system of cycles $t_1-t_0$.
A case by case analysis shows that suitable surfaces can always be constructed.
\end{proof}

\begin{example}
\label{example:twist666}
Let $R$ be the $6\times 6\times 6$ torus.
Let $t_0 = t_\basetiling$, $t_1$, $t_2$ and $t_3$
be as in Figure \ref{fig:666f0}.
Recall from Example \ref{example:modulus} that 
the modulus for these tilings is $m = 0$.
We have $\TW(t_1,t_2) = +2 \ne 0 \in \ZZ$:
it follows from Proposition \ref{prop:fliptwist}
that $t_1$ and $t_2$ are not connected to each other (or to $t_0$)
by a sequence of flips.
A similar computation verifies that $\TW(t_3,t_0) = 0$.
Indeed, there exists a sequence of flips 
joining $t_0$ and $t_3$.
Experiments indicate that for almost all tilings $t$ of $R$
with $\Flux(t) = 0$ and $\TW(t,t_0) = 0$
there exists a sequence of flips joining $t_0$ and $t$.

Let $t_4$ and $t_5$ be as in Figure \ref{fig:666f1}.
Again from Example \ref{example:modulus},
the modulus for these tilings is $m = 2$.
We have $\TW(t_5,t_4) = +2 = 0 \in \ZZ/2\ZZ$.
Consistent with Proposition \ref{prop:fliptwist},
a computation shows that $t_4$ and $t_5$
are connected by a finite sequence of flips.
Again, experiments indicate that for almost all tilings $t$ of $R$
with $\Flux(t) = \Flux(t_4)$ and $\TW(t,t_4) = 0 \in \ZZ/2\ZZ$
there exists a sequence of flips joining $t_4$ and $t$.
\end{example}

Let $\gamma_1$ and $\gamma_2$ be disjoint systems
of smooth cycles in an oriented $3$-manifold $R$.
If $[\gamma_1] = [\gamma_2] = 0 \in H_1(R)$, then
there  exist Seifert surfaces $S_1$ and $S_2$
for $\gamma_1$ and $\gamma_2$.  Classically, the \emph{linking number}
$\Link(\gamma_1;\gamma_2) = \Link(\gamma_2;\gamma_1) \in \ZZ$
of $\gamma_1$ and $\gamma_2$
is defined as the number of intersections
(with sign) between $\gamma_1$ and $S_2$
(or $\gamma_2$ and $S_1$).  Furthermore, the linking number is 
independent of the choice of $S_1$ and $S_2$. 
For our more general spaces,
the linking number must be considered
with respect to the modulus $m$ of the tiling.  
Then, the linking number quantifies the difference in twist.  
Namely, suppose $R$ is a cubiculated region and
$t_0$, $t_1$, $t_2$ and $t_3$ are tilings of $R$ with equal Flux. 
If the systems of cycles
$\gamma_1 = t_1 - t_0 = t_3 - t_2$ and
$\gamma_2 = t_2 - t_0 = t_3 - t_1$ are disjoint then 
$$\TW(t_3;t_2) - \TW(t_1;t_0) =
\TW(t_3;t_1) - \TW(t_2;t_0) = 2\Link(\gamma_1;\gamma_2)$$
 and $ 
\TW(t_3;t_0) = \TW(t_3;t_2) + \TW(t_2;t_0) 
= \TW(t_3;t_1) + \TW(t_1;t_0).$
More generally, if $t_2 - t_1$ and $t_1 - t_0$ are not disjoint,
then refine and slightly move these
systems of cycles using flips (and Proposition~\ref{prop:fliptwist})
to obtain disjoint cycles.
Together this gives the following. 

\begin{prop}
\label{prop:generaladditivetwist}
Let $R$ be a cubiculated region.
Let $t_0$, $t_1$ and $t_2$ be tilings of $R$ with
$\Flux(t_0) = \Flux(t_1) = \Flux(t_2)$.
Then $\TW(t_2;t_0) = \TW(t_2;t_1) + \TW(t_1;t_0)$.
\end{prop}

We are now ready to define the twist of a tiling.

\begin{definition}
\label{def:twist}
Let $R$ be a cubiculated region. 
For any possible value $\Phi$ of the Flux of a tiling, 
choose a base tiling $t_{\Phi}$.
For a tiling  $t$  of a region $R$ with $\Flux(t) = \Phi$
define 
\begin{equation*}
\label{eq:twist}
\Tw(t) := \phi(t;t-t_{\Phi}) = \phi(t_{\Phi};t-t_{\Phi}) \in \ZZ/m\ZZ.
\end{equation*}
\end{definition}

\begin{example}
\label{example:Twphi}
Let $R$ be the $4\times 4\times 4$ box;
let $t_0$ and $t_1$ be as in Example \ref{example:phi};
let $t_{\basetiling}$ be the tiling with vertical dominoes.
We have $\Tw(t_0) = 1$ and $\Tw(t_1) = 0$,
consistently with Example \ref{example:TWphi}.
\end{example}

\begin{example}
\label{example:twist666f0}
Let $t_0$, $t_1$, $t_2$ and $t_3$ be the tilings
of the $6\times 6\times 6$ torus
shown in Figure \ref{fig:666f0}.
Take $t_{\basetiling} = t_0$ so that
$\Phi = \Flux(t_0) = \Flux(t_1) = \Flux(t_2) = \Flux(t_3) = 0$.
Take $t_{\Phi} = t_0$: we then have
$\Tw(t_0) = \Tw(t_3) = 0$, $\Tw(t_1) = +1$ and $\Tw(t_2) = -1$.
\end{example}

\begin{coro}
\label{coro:trittwist}
If $t_1$ is obtained from $t_0$ by a positive trit then
$\Flux(t_1) = \Flux(t_0)$ and $\Tw(t_1) = \Tw(t_0) + 1$.
\end{coro}

\section{Height Functions}
\label{sec:coquadriculated}

In this section we consider discrete Seifert surfaces 
together with the restriction of $\cG(R^\ast)$ to the surface.  
We prove that, under suitable hypothesis,
two tilings of such a surface are connected by flips.
The proof relies on a development of \emph{height functions} 
appropriate to our setting.
Height functions are a standard tool in the study of domino tilings,
see e.g.~\cite{thurston1990},  \cite{conway1990tiling}, \cite{davidtomei}.

If $S$ is a discrete Seifert surface in $R$,
a tiling of $R$ which is tangent to $S$
restricts to a tiling of $S$, i.e.,
a matching of $\cG^\ast$, the restriction of $\cG(R^\ast)$ to $S$.
Let $\cT = \cT(S,\cG^\ast)$ be the set of all tilings of $S$.

Just as for $3$-dimensional dominoes, if $t_0, t_1 \in \cT$,
the difference $t_1 - t_0$ can be seen as a system of cycles in $S$.
We interpret $[t_1 - t_0]$ to be an element of $H_1(S)$ and,
given a choice of a base tiling $t_\basetiling \in \cT$,
set $\Flux_S(t) = [t - t_\basetiling] \in H_1(S)$.
Given $a \in H_1$, let $\cT_a$
be the equivalence class of tilings of $S$ with the same Flux $a$.
A tiling $t \in \cT$ is \emph{stable}
if every edge of $\cG^\ast$ belongs to some tiling 
in the equivalence class $\cT_{\Flux_S(t)}$; in this case,
we also call the set $\cT_{\Flux_S(t)}$ \emph{stable}.

\begin{example}
\label{example:stable}
Consider the surface $S$ in Figure \ref{fig:stable1}.
The surface is shown both as a set of cubes in $R$
and as a surface in $R^\ast$.
We also show a related graph to be constructed below.

\begin{figure}[ht]
\centerline{\includegraphics[scale=0.25]{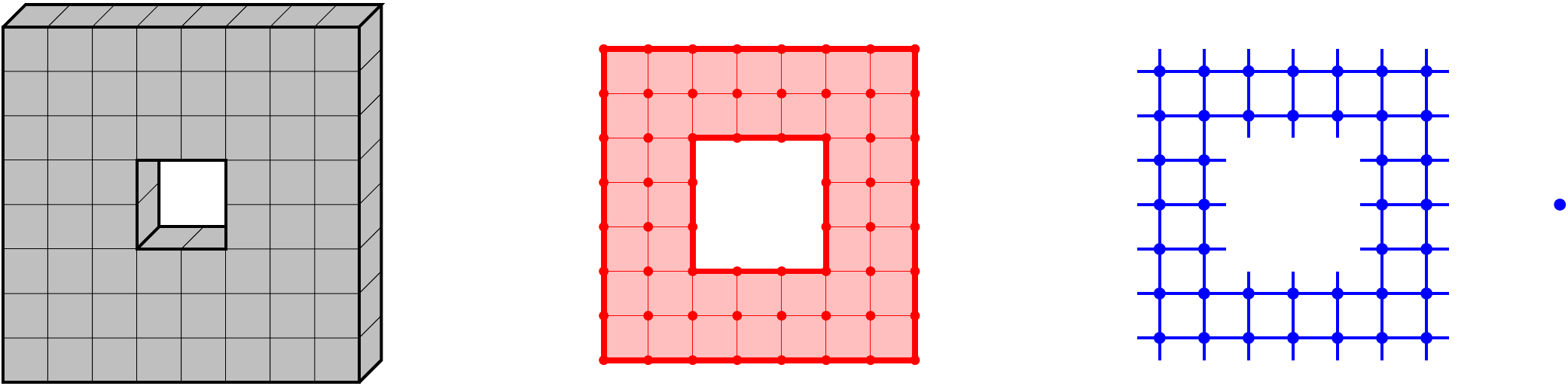}}
\caption{A surface shown in $3$ different ways.
The first one is as a set of cubes in $R$.
The second one is as a surface in $R^\ast$.
The third one is the graph with vertices in $V$.}
\label{fig:stable1}
\end{figure}

Figure \ref{fig:stable2} shows three tilings $t_1$, $t_2$ and $t_3$
of the surface $S$.
The two tilings $t_1$ and $t_2$ have the same flux and
there are $556515$ tilings in the class $\cT_{\Flux_S(t_1)}$.
It is not hard to verify that this set is stable
(every domino in $S$ belongs to a tiling in this class).
The tiling $t_3$ is the only one in its class,
which is thus unstable.
\end{example}

\begin{figure}[ht]
\centerline{\includegraphics[scale=0.25]{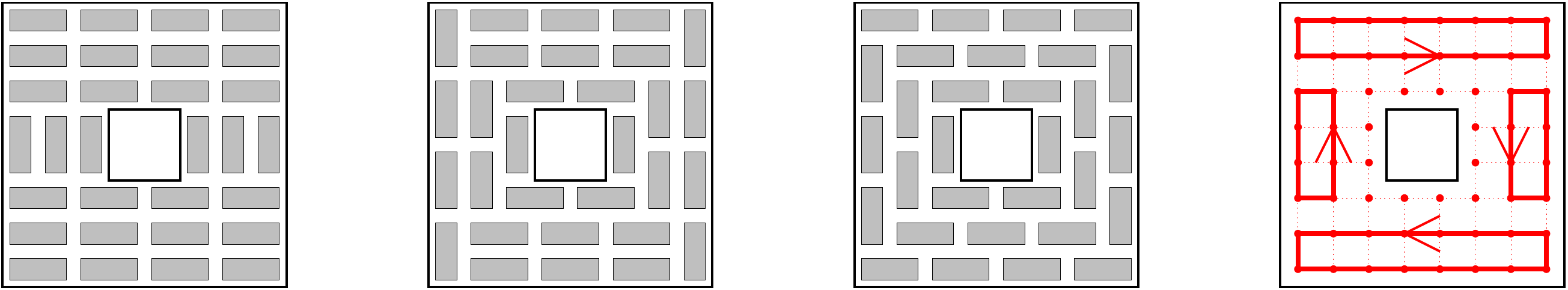}}
\caption{Three tilings of the surface $S$ shown in Figute \ref{fig:stable1}.
The difference between the first two tilings is also shown.}
\label{fig:stable2}
\end{figure}

When taking refinements, we may use $t_\basetiling'$
as a base tiling of refined $S$ 
(after a few flips to make sure the tiling is tangent to $S$).
It is then easy to see that $\Flux_S(t') = \Flux_S(t)$
(we will blur the distinction between a surface $S$ and its refinement $S'$).
Furthermore, sufficient refinement makes any class stable:
the condition of existence of a suitable tiling becomes easy after refinement.

Let $a = \Flux_S(t_0) = \Flux_S(t_1)$ so that $t_0, t_1 \in \cT_a$.
Let $V$ be the set of components of $S \smallsetminus \cG^\ast$
(squares in $R^\ast$)
with one extra object called $\infty$ corresponding to $\partial S$.
For each oriented edge $e$ of $\cG^\ast$, there is an element $e_l \in V$
to its left and an element $e_r \in V$ to its right;
if the edge is contained in $\partial S$ then one of these is $\infty$.
Two elements $v_0, v_1 \in V$ are \emph{neighbors}
if there exists an oriented edge $e$ with $v_0 = e_l$ and $v_1 = e_r$.
We thus obtain a graph, as shown in Figure \ref{fig:stable1};
the vertex at infinity is shown separately
and the edges from it are not shown.

Let $C_2$ (relative to $R^\ast$)
be the $\ZZ$-module of functions $w: V \to \ZZ$.
Let $C_1$ be the $\ZZ$-module spanned by oriented edges of $\cG^\ast$.
Define the boundary map $\partial: C_2 \to C_1$ as follows:
given $w \in C_2$ and $e$ an oriented edge of $\cG^\ast$,
the coefficient of $e$ in $\partial w$ is $w(e_l) - w(e_r)$.
Let $B_1 \subseteq C_1$ be the image of $\partial: C_2 \to C_1$:
given $g \in B_1$, there is a unique element $w = \wind(g) \in C_2$
with $w(\infty) = 0$ and $\partial w = g$.
We call $\wind(g)$ the \emph{winding} of $g$.

\begin{example}
\label{example:winding}
For the surface $S$ shown in Figure \ref{fig:stable1},
the set $V$ has $41$ elements:
the $40$ squares in $R^\ast$ and the vertex $\infty$ (shown separately).

\begin{figure}[ht]
\centerline{\includegraphics[scale=0.25]{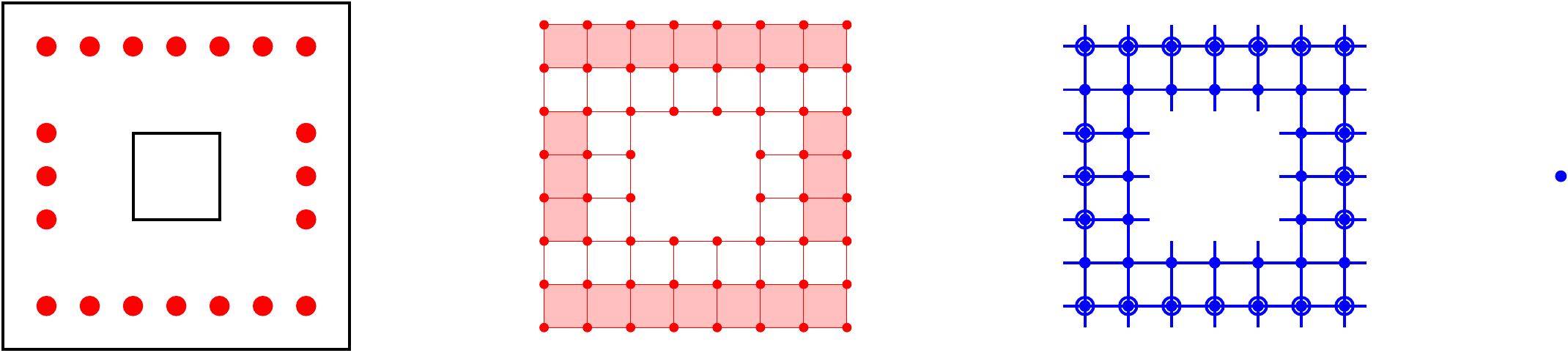}}
\caption{The winding $\wind(t_2-t_1)$.}
\label{fig:winding}
\end{figure}

Let $t_1$ and $t_2$ be the first two tilings in Figure \ref{fig:stable2}.
Their difference $g = t_2 - t_1 \in B_1$ is shown
in the fourth figure of Figure \ref{fig:stable2}.
The winding $\wind(g)$ is shown in three different ways
in Figure \ref{fig:winding}:
it has value $1$ in the indicated squares
and $0$ elsewhere (including in $\infty$).
\end{example}

Given a tiling $t \in \cT_a$, we proceed to contruct
the \textit{height function} $h_t: V \to \RR$
in Equation \ref{equation:heightfunction}.
In order to do that, we use the winding functions
$\wind(t-\tilde t): V \to \ZZ$ for each $\tilde t \in \cT_a$.
Notice that given two tilings $t, \tilde t \in \cT_a$,
we have $g = t - \tilde t \in B_1$.
The function $w = \wind(g): V \to \ZZ$
satisfies $w(\infty) = 0$.
Furthermore, given a black-to-white edge $e$,
we have $w(e_l) - w(e_r) = [e \in t_1] - [e \in t_0]$
(we use Iverson notation:
$[e \in t]$ equals $1$ if $e \in t$ and $0$ otherwise).
In particular, if $v_0, v_1 \in V$ are neighbors then
$|w(v_0) - w(v_1)| \le 1$.
We are now ready to construct the \emph{height function}
$h_t: V \to \RR$:
\begin{equation}
\label{equation:heightfunction}
h_t = \frac{1}{|\cT_a|} \sum_{\tilde t \in \cT_a} \wind(t - \tilde t). 
\end{equation}
Thus, $h_t$ is the average of the windings $\wind(t-\tilde t)$
for $\tilde t \in \cT_a$.
Notice that for any $t \in \cT_a$ the height function $h_t$ satisfies:
\begin{enumerate}
\item[(a)]{$h_t(\infty) = 0$;}
\item[(b)]{for any $\tilde t \in \cT_a$ and any $v \in V$ we have
$h_{t}(v) \equiv h_{\tilde t}(v) \pmod \ZZ$;}
\item[(c)]{if $v_0, v_1 \in V$ are neighbors then $|h_t(v_0) - h_t(v_1)| < 1$.}
\end{enumerate}
Condition (a) follows from the equivalent equation for $\wind(t_1-t_0)$.
Condition (b) follows from $h_{t} - h_{\tilde t} = \wind(t-\tilde t)$.
Finally, the strict inequality in (c) follows from the hypothesis
of stability of $\cT_a$.

\begin{example}
\label{example:height}
This definition of height function is good
for theoretical use but computationally is not very practical.
Moreover, the values are complicated rational numbers
(unlike the situation in \cite{thurston1990}).
Consider the class $\cT_{\Flux_S(t_1)}$ in Example \ref{example:stable}.
In vertex $v = (1,1)$ (the nontrivial vertex furthest to the top left)
the height function assumes the values $h_{t_1}(v) = -224401/556515$
and $h_{t_2}(v) = 332114/556515$.
These are the two possible values of $h_t(v)$ for $t \in \cT_{\Flux_S(t_1)}$:
notice that they differ by $1$.
For $t \in \cT_{\Flux_S(t_1)}$,
the value of $h_t(v)$ is determined
by the position of the top left domino in the tiling $t$.
\end{example}

Conversely, any function $h: V \to \RR$ satisfying conditions
(a), (b) and (c) above is the height function of a (unique) tiling
$t \in \cT_a$.
Recall that tilings $t \in \cT_a$ can be interpreted as
elements of $C_1$:
for a positively oriented edge $e$,
the coefficient of $e$ in $t$ is $[e \in t]$.
Given $h$ and a positively oriented edge $e$,
we have $e \in t$ if and only if $h(e_l) > h(e_r)$.
Conditions (b) and (c) guarantee
that for any vertex in $\cG^\ast$ exactly one edge adjacent to it
receives the coefficient $1$.

We give another description of the method to obtain
a tiling $t$ given its height function $h$.
Take a base tiling $t_0 \in \cT_a$:
we may also consider $t_0 \in C_1$.
Set $w = h - h_{t_0}$ (which is integer valued)
and $t = t_0 + \partial w \in C_1$.

As for tilings of planar regions,
it follows from this characterization of height functions that
the maximum or minimum of two height functions
is a height function.
For $t, \tilde t \in \cT_a$, write $t \le \tilde t$
if $h_{t}(v) \le h_{\tilde t}(v)$ for all $v \in V$.
Also, $t < \tilde t$ if $t \le \tilde t$ and $t \ne \tilde t$.

Two tilings $t, \tilde t \in \cT_a$ differ by a flip
at $v \in V \smallsetminus \{\infty\}$ if
$h_t(v) - h_{\tilde t}(v) = \pm 1$ and
$h_t(\tilde v) = h_{\tilde t}(\tilde v)$
for $\tilde v \in V \smallsetminus \{v\}$.
Conversely, a flip at $v \in V \smallsetminus \{\infty\}$
is allowed from $t \in \cT_a$
if and only if $v$ is a local maximum or minimum of $h_t$.

\begin{theo}
\label{theo:coquadriculated}
Consider a pair $(S,\cG^\ast)$
and two stable tilings $t_0$, $t_1$ of $S$
with $\Flux_S(t_0) = \Flux_S(t_1)$.
Then $t_0$ and $t_1$ are connected by flips.
\end{theo}

\begin{proof}
Assume $t_0 < t_1$.
We show that we can perform a flip on $t_1$
in order to obtain $\tilde t$ with $t_0 \le \tilde t < t_1$;
by induction, this completes the proof.
Let $V_1 \subset V$ be the set of $v \in V$ for which
$h_{t_1}(v) - h_{t_0}(v)$ is maximal.
Let $v_2$ be the point of $V_1$ where $h_{t_1}$ is maximal.
We claim that $v_2$ is a local maximum of $h_{t_1}$.
Indeed, let $\tilde v$ be a neighbor of $v_2$.
If $\tilde v \in V_1$ then $h_{t_1}(\tilde v) < h_{t_1}(v_2)$
by definition of $v_2$.
If $\tilde v \notin V_1$ then
$h_{t_0}(v_2) < h_{t_0}(\tilde v) \le h_{t_1}(\tilde v) < h_{t_1}(v_2)$
by definition of $V_1$, completing the proof of the claim.
Let $h_{\tilde t}: V \to \RR$ be defined by
\[ h_{\tilde t}(v) = \begin{cases} h_{t_1}(v), & v \ne v_2; \\
h_{t_1}(v_2) - 1, & v = v_2. \end{cases} \]
The function $h_{\tilde t}$ satisfies conditions (a), (b) and (c)
and therefore defines a valid tiling $\tilde t \in \cT_a$
with the required properties.
\end{proof}

\begin{remark}
The results of this section, including Theorem~\ref{theo:coquadriculated},
hold more generally for any \emph{coquadriculated surface},
a pair $(S,\cG^\ast)$ consisting of:
an oriented compact connected topological surface $S$
with non-empty boundary $\partial S$ and
an embedded connected bipartite graph  $\cG^\ast \subset S$
such that every connected component
of $S \smallsetminus \cG^\ast$ is a square,
that is, is surrounded by a cycle of length $4$ in $\cG^\ast$. 
Namely, it is not necessary to have an ambient manifold
to induce the surface being tiled.  
\end{remark}

\section{Connectivity}
\label{sec:proof}

In this section we prove our main theorem.  
Let $R$ be a cubiculated region.
Let $t_0$, $t_1$ be tilings of $R$ with $\Flux(t_0) = \Flux(t_1)$.
A discrete Seifert surface $(S,\psi)$
for the pair $(t_0,t_1)$ is:
\begin{itemize}
\item{ \emph{balanced} if the number of black vertices  equals the number of white vertices in the interior
of $S$;}
\item{ \emph{zero-flux} if $\phi(t_0;S) = \phi(t_1;S) = 0$; }
\item{ \emph{tangent} if $t_0$ and $t_1$ are tangent to $S$
(including in the interior of the surface).  }
\end{itemize}
Notice that a tangent surface is clearly both balanced and zero-flux.
The converse implications in general do not hold.

Recall from Lemma \ref{lemma:existseifert} that if $\Flux(t_0) = \Flux(t_1)$
then there exists a refinement $R^{(k)}$
and a discrete Seifert surface for the pair $(t_0,t_1)$ in this refinement.
We are now ready to prove our main theorem by considering when nicer Seifert surfaces can be constructed.

\begin{proof}[Proof of Theorem \ref{theo:main}]
We in fact prove a series of results.
Each item below could be considered a distinct lemma,
but since they build naturally to the main result, 
 we prefer to present this as a single proof.
   
Consider a cubiculated region $R$
and two distinct tilings $t_0$ and $t_1$ of $R$
with $\Flux(t_0) = \Flux(t_1)$.
\begin{enumerate}
\item{If $\Tw(t_0) = \Tw(t_1)$
then there exists a refinement $R^{(k)}$
and a discrete \emph{zero-flux} Seifert surface
for the pair $(t_0,t_1)$ in this refinement.}
\item{If there exists  a discrete zero-flux Seifert surface
for the pair $(t_0,t_1)$ then
there exists a refinement $R^{(k)}$
and a discrete \emph{balanced zero-flux} Seifert surface
for the pair $(t_0,t_1)$ in this refinement.}
\item{If there exists  a discrete balanced zero-flux Seifert surface
for the pair $(t_0,t_1)$ then
there exists a refinement $R^{(k)}$ and
tilings $\tilde t_i$ of $R^{(k)}$ with
$\tilde t_1 - \tilde t_0 = (t_1 - t_0)^{(k)}$,
$\tilde t_i$ obtained from $t_i^{(k)}$ by a sequence of flips,
and a \emph{tangent} Seifert surface
for the pair $(\tilde t_0,\tilde t_1)$.}
\item{If there exists  a discrete tangent Seifert surface
for the pair $(t_0,t_1)$ then
there exists a refinement $R^{(k)}$
for which there exists a sequence of flips taking $t_0^{(k)}$ to $t_1^{(k)}$.}
\end{enumerate}

For item (1), start by constructing a smooth Seifert surface $S$.
Let $m = \mu(\Flux(t_0))$ be the modulus of $t_0$ and $t_1$
(see Definition \ref{definition:modulus}).
By the definition of twist, $\phi(t_0;S) = \phi(t_1;S)$
is a multiple of $m$ and therefore an integer linear combination
of $\phi(t_0;S_i) = \phi(t_1;S_i)$ where the $S_i$ are smooth closed surfaces:
\[ \phi(t_0;S) + \sum c_i \phi(t_0;S_i) = 0. \]
Let $\tilde S$ be the union of $S$ with $c_i$ copies of $S_i$
(reverse orientation if $c_i < 0$).
Perturb the surfaces to guarantee transversality.
Along intersection curves, perform a standard cut-re-glue-smoothen
procedure (see e.g.\cite{Lickorish})
to obtain a smooth embedded Seifert Surface
$\hat S$ with $\phi(t_0;\hat S) = \phi(t_0;\tilde S)$.
As in Lemma \ref{lemma:existseifert},
after refinements, $\hat S$ can be approximated by a discrete surface.

For item (2), notice first that the parity of the number of black 
and white interior vertices  is already the same,
otherwise $\phi(t_0;S)$ would be odd.
In order to increase the black-minus-white difference by two,
look for a white vertex in a planar part of $S$
and lift a surrounding square as in Figure~\ref{fig:lift}.
Repeat the procedure as needed.
\begin{centering}
\begin{figure}[ht]
\centerline{\includegraphics[width=4in]{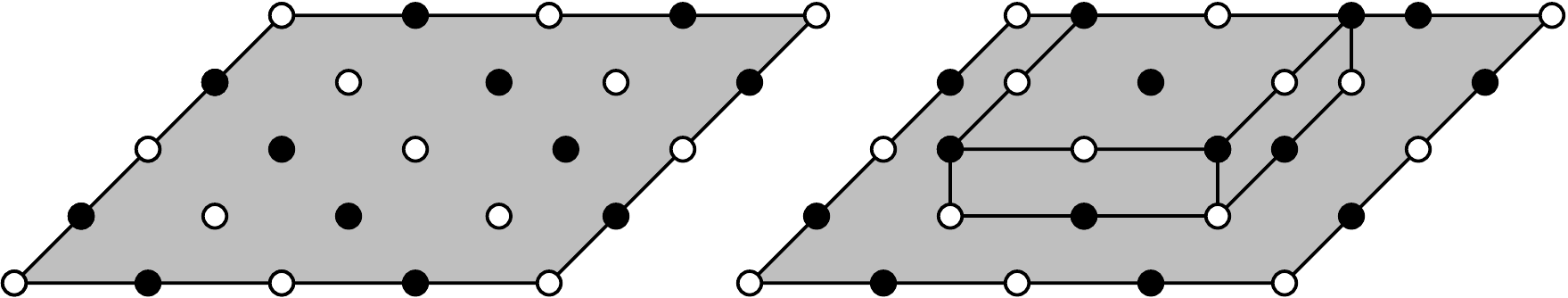}}
\caption{The four center squares on the bottom layer are removed
and the surfaces is lifted. 
This increases the difference between
the number of black versus white vertices;
there is a net increase of $5$ black vertices and $3$ white vertices.}
\label{fig:lift}
\end{figure}
\end{centering}

For item (3), consider vertices in the interior of $S$
which are matched with points outside $S$.
Classify these vertices as black-above, white-above,
black-below and white-below according to their color
and the position of their match.
It follows from $S$ being balanced and zero-flux that
the number of black-above and white-above are equal
(and similarly for black-below and white-below).
Associate vertices black-above and white-above in pairs
(and similarly for black-below and white-below).
We show how to perform flips in order to have such pairs cancel out.
Starting from a black-above vertex, draw a simple curve $\gamma$ on $S$
using tangent dominoes until you arrive at its white-above partner.
Since the surface has been refined there are large regions
of parallel dimers (or dominoes) both on the surface and near $S$. 
We may therefore construct a narrow disk $S_1$ by going ``above'' $\gamma$
in the direction normal to $S$.
As in Figure~\ref{fig:refine}, 
by performing flips and taking refinements,
we may assume $S_1$ to be a surface with an induced tiling by $t$.
By construction,
both the black-above and the white-above dominoes belong
to $\gamma_1 = \partial S_1$.
Again by virtue of refinements
we may assume $t$ to be stable (as a tiling of $S_1$).
Apply Theorem \ref{theo:coquadriculated}
(for $S_1$, not the original $S$)
in order to obtain a sequence of flips
whose effect is to rotate $\gamma_1$,
thus getting rid of both the black-above and the white-above dominoes.

For item (4), we apply Theorem \ref{theo:coquadriculated}
to the surface $S$.
If $t_0$ (and $t_1$) are stable, this can be done directly.
Otherwise, take refinements:
appropriate refinements of $t_0$ (and $t_1$) are stable.
This completes
the proof of item (b) of Theorem \ref{theo:main}.

As to item (a), let $l$ be a positive integer such that
$\Tw(t_1) = \Tw(t_0) \pm l$.
Apply refinements such that $t_0^{(k)}$
and $t_1^{(k)}$ contain at least $l$ boxes
of dimension $3\times 3\times 2$
tiled by $9$ parallel dominoes.
Starting from $t_0^{(k)}$, apply flips and trits 
inside $l$ such boxes in such a manner that
the twist increases or decreases by $l$.
This connects by flips and trits
$t_0^{(k)}$ to a tiling $t$ with $\Tw(t) = \Tw(t_1)$.
By item (2), $t$ and $t_1$ can be connected by flips
(possibly after further refinement).
We thus completed the proofs of both items (a) and (b).
\end{proof}

\section{Final Remarks}

Theorem~\ref{theo:main} is the first positive result concerning
connectivity of three-dimensional domino tilings.  Given the nature of
the result, many natural questions arise.

Our methods rely heavily on refinements
but it is important to point out that they are not simply an artifact
of our proof techniques.  For example, already in the $4 \times 4
\times 4$ box, there are tilings with the same twist which are not
flip connected before refinement.

On the other hand,
allowing for an arbitrary number of refinements appears to be overkill. 
In all known examples, related but far more modest operations are sufficient. 
For instance, in the general case of boxes, it is not known if
refinements are needed: If $t_0$ and $t_1$ are two tilings of a fixed
box, can $t_0$ and $t_1$ be connected by flips and trits?  It is known
that refinements (or related operations)
are necessary for connectivity of tilings of simple
prisms and tori~\cite{primeiroartigo}.
For boxes which are not too narrow,
recent results \cite{saldanha2019}
show that adding a little vertical room is already sufficient.

It is also unknown how often refinement is necessary for certain regions. 
Recent results \cite{saldanha2021}
show that for boxes which are not too narrow,
equal twist almost always implies
connectivity via flips.
But many other cases remain open.
For example, consider the cubical torus given by $\ZZ^3/(N \cdot \ZZ^3)$
for some even integer $N$.
Based on computational experiments,
we conjecture that the probability that two tilings with
the same flux and twist are connected by flips tends to $1$ as $N$
tends to infinity.

More broadly, one asks, what does a `typical' tiling look like?  The
twist essentially partitions the space of tilings into flip connected
components.  For a fixed region, what is the distribution of the
twist?  For example, consider again the case of boxes.  Let the base
tiling (twist $0$) be one in which all tiles are parallel to a
fixed axis. 
In certain special cases, the twist is normally distributed about $0$
(see~\cite{saldanha2021}); is this true in greater generality?

Finally, we mention that the case of higher dimension needs to be explored.
For boxes, twist can be defined in $\ZZ/(2)$ (but not in $\ZZ$).
For certain boxes, the set of tilings has
two large connected components via flips
of approximately the same size
(corresponding to the values $0$ and $1$ of the twist)
and an unknown number of very small components \cite{klivanssaldanha}.

\bibliography{biblio}{}
\bibliographystyle{plain}

\bigskip

{\raggedright
\noindent
Departamento de Matem\'atica, PUC-Rio \\
Rua Marqu\^es de S\~ao Vicente, 225, Rio de Janeiro, RJ 22451-900, Brazil \\
\url{jufreire@gmail.com}\\
\url{pedrohmilet@gmail.com}\\
\url{saldanha@puc-rio.br}\\

\smallskip

\noindent Division of Applied Mathematics and Department of Computer Science\\
Brown University, Providence, RI, USA\\
\url{caroline_klivans@brown.edu}\\
}

\end{document}